\documentclass{amsart}
\usepackage{amsfonts,amssymb,amscd,amsmath,enumerate,verbatim,calc}
\usepackage[all]{xy}
\usepackage{euscript}

\newcommand{\CM}{Cohen-Macaulay}

\newcommand{\fA}{\mathfrak{a}}
\newcommand{\eC}{\EuScript{C}}

\newcommand{\Bf}{\mathbf{f} }
\newcommand{\bx}{\mathbf{x} }
\newcommand{\n}{\mathfrak{n} }
\newcommand{\m}{\mathfrak{m} }

\newcommand{\q}{\mathfrak{q} }

\newcommand{\Z}{\mathbb{Z} }
\newcommand{\Q}{\mathbb{Q} }
\newcommand{\cV}{\mathcal{V}}
\newcommand{\T}{\mathcal{T} }
\newcommand{\N}{\mathbb{N} }
\newcommand{\R}{\mathcal{R} }

\newcommand{\Cc}{\mathcal{C} }

\newcommand{\Kc}{\mathcal{K} }

\newcommand{\Fb}{\mathbb{F} }
\newcommand{\Gb}{\mathbb{G} }

\newcommand{\rt}{\rightarrow}

\newcommand{\ov}{\overline}

\newcommand{\Om}{\Omega}

\newcommand{\wt}{\widetilde }

\newcommand{\image}{\operatorname{image}}
\newcommand{\Tr}{\operatorname{Tr}}
\newcommand{\codim}{\operatorname{codim}}

\newcommand{\depth}{\operatorname{depth}}

\newcommand{\ann}{\operatorname{ann}}
\newcommand{\cx}{\operatorname{cx}}
\newcommand{\rank}{\operatorname{rank}}
\newcommand{\curv}{\operatorname{curv}}
\newcommand{\Ass}{\operatorname{Ass}}
\newcommand{\htt}{\operatorname{height}}

\newcommand{\proj}{\operatorname{proj}}
\newcommand{\Spec}{\operatorname{Spec}}

\newcommand{\CMS}{\operatorname{\underline{CM}}}

\newcommand{\CMa}{\operatorname{CM}}

\newcommand{\Gau}{\underline{\Gamma}}
\newcommand{\projdim}{\operatorname{projdim}}
\newcommand{\Hom}{\operatorname{Hom}}
\newcommand{\md}{\operatorname{mod}}
\newcommand{\mds}{\operatorname{\underline{mod}}}

\newcommand{\sHom}{\operatorname{\underline{Hom}}}
\newcommand{\Ext}{\operatorname{Ext}}
\newcommand{\Tor}{\operatorname{Tor}}

\theoremstyle{plain}

\newtheorem{theorem}{Theorem}[section]
\newtheorem{corollary}[theorem]{Corollary}
\newtheorem{lemma}[theorem]{Lemma}
\newtheorem{proposition}[theorem]{Proposition}

\newtheorem{conjecture}[theorem]{Conjecture}

\theoremstyle{definition}
\newtheorem{definition}[theorem]{Definition}
\newtheorem{s}[theorem]{}
\newtheorem{remark}[theorem]{Remark}

\newtheorem{construction}[theorem]{Construction}
\theoremstyle{remark}

\begin{document}

\title[Stable category]{On the Stable category of maximal Cohen-Macaulay modules over Gorenstein rings }
\author{Tony~J.~Puthenpurakal}
\date{\today}
\address{Department of Mathematics, IIT Bombay, Powai, Mumbai 400 076}

\email{tputhen@math.iitb.ac.in}

 \begin{abstract}
Let $(A,\m), (B,\n) $ be  Gorenstein local rings  and let $\CMS(A)$ be its stable category of finitely generated maximal \CM \ $A$-modules. Suppose $\CMS(A) \cong \CMS(B)$ as triangulated categories. Then we show
\begin{enumerate}
  \item  If $A$ is an abstract complete intersection of codimension $c$ then so is $B$.
  \item If $A, B$ are Henselian and not hypersurfaces then $\dim A = \dim B$.
  \item  If $A$ (and so $B$)  are complete hypersurface  singularity and multiplicity of $A$ is at least three then $\dim A - \dim B$ is even.
  \item If $A, B$ are Henselian and $A$ is an isolated singularity then so is $B$.
\end{enumerate}
We also give some applications of our results.  It should be remarked that
if
 $R,S$ are complete \CM \ but not necessarily Gorenstein  and if there is an triangle isomorphism between the singularity categories of $R$ and $S$ then it is possible that $\dim R - \dim S$ is odd, see  M.~Kalck; Adv. Math. 390 (2021), Paper No. 107913.
\end{abstract}
 \maketitle
\section{introduction}

Representation theory of Artin Algebras is a well established branch of mathematics. Auslander discovered that many concepts in representation theory of Artin algebras have
natural analogues in the study of maximal \CM \ ( = MCM)  modules of a commutative \CM \  local ring $A$. See \cite{Y} for a nice exposition of these ideas. In this paper we study analogues of a natural concept in theory of Artin algebras in the study of MCM modules over commutative Gorenstein local ring.

Let $R$ be a commutative Artin ring and let $\Lambda$ be a not-necessarily commutative Artin $R$-algebra. Let $\mds(\Lambda)$ denote the stable category of $\Gamma$.
The study of equivalences of stable categories of Artin $R$-algebras has a rich history; see \cite[Chapter X]{ARS}.  By following Auslander's idea  we investigate equivalences of the stable
category of MCM modules over commutative \CM \ local rings. If $A$ is Gorenstein local then the stable category of MCM $A$-modules has a triangulated structure. So as a first
iteration in this program, in this paper we investigate triangle equivalences of stable categories of MCM modules over commutative Gorenstein rings.

 In general,  stable category of Artin algebras (over a commutative Artin ring) need not be triangulated. So by an equivalence we just mean as equivalence of additive categories over a commutative Artin ring.  If $\Gamma$ is a self-injective Artin $R$-algebra then $\mds(\Gamma)$ is triangulated.  Furthermore if
$f \colon \mds(\Gamma) \rt \mds(\Lambda)$ is an equivalence of additive categories
then $f$ commutes with the shift functor on the corresponding triangulated categories; see \cite[1.12, Chapter X]{ARS}. However for the best  of our knowledge  this isomorphism is \emph{NOT
natural}. So it does not follow that $f$ is an equivalence of triangulated categories.

If two finite dimensional self-injective algebra's over a field are derived equivalent then they are also stably equivalent by a result of Rickard, see \cite[2.2]{R2}.  Another result of Rickard's shows that if two rings are derived equivalent  then their center's are isomorphic, see \cite[9.2]{R1}. Thus derived equivalences
of commutative rings are not interesting.

Let $(A,\m)$ be a  commutative Gorenstein local ring with residue field $k$. Let $\CMa(A)$ denote the full subcategory of finitely generated  MCM $A$-modules  and let $\CMS(A)$ denote the stable category of MCM $A$-modules. Recall that objects in $\CMS(A)$ are same as objects in $\CMa(A)$. However the set of morphisms $\underline{\Hom_A}(M,N)$ between $M$ and $N$ is  $= \Hom_A(M,N)/P(M,N)$ where $P(M,N)$ is the set of $A$-linear maps from $M$ to $N$ which factor through a finitely generated free module. It is well-known that $\CMS(A)$ is  a triangulated category with translation functor $\Omega^{-1}$,  (see \cite{Bu}; cf. \ref{T-struc}).  Here $\Omega(M)$ denotes the syzygy of $M$ and $\Omega^{-1}(M)$ denotes the co-syzygy of $M$. Also recall that an object $M$ is zero in $\CMS(A)$ if and only if it is free  considered as an $A$-module. Furthermore $M \cong N$ in $\CMS(A)$ if and only
if there exists finitely generated free modules $F,G$ with $M\oplus F \cong N \oplus G$ as $A$-modules.  If $A$ is regular local then all MCM
modules are free. So $\CMS(A) = 0$.

We use Neeman's book \cite{N} for notation on triangulated categories. However we will assume that if $\mathcal{C}$ is a triangulated category then $\Hom_\mathcal{C}(X, Y)$ is a set for any objects $X, Y$ of $\mathcal{C}$.

Let $R$ be a commutative  Noetherian  ring. Let $D^b(R)$ be the bounded derived category of $R$. A complex $X$  of $R$-modules is said to  be perfect if it is isomorphic in $D^b(R)$ to  a bounded complex of finitely generated projective $R$-modules. The subcategory $D^b_{perf}(R)$ of perfect complexes  is equal to $[R]$ the thick hull of $R$ considered as a subcategory of $D^b(R)$ (see \cite[1.2.1]{Bu}). The quotient category $D^b(R)/[R]$ is called the singularity category of $R$ and is denoted by $D_{sg}(R)$. Note if $R$ is  of finite Krull dimension then $D_{sg}(R) = 0$ if and only if $R$ is a regular ring. Buchweitz proved that if $(R,\n)$ is a Gorenstein local ring then $\CMS(R) \cong D_{sg}(R)$ (see \cite[4.4.1]{Bu}).

\textbf{I} \emph{ Properties  invariant under stable equivalences:}

\textbf{I (a):}    We say a local ring $A$ is a \emph{geometric complete intersection}  if $A = Q/(f_1,\ldots, f_c)$ where $Q$ is regular local and $f_1,\ldots, f_c$ is a $Q$-regular sequence. We say $A$ is an \emph{abstract complete intersection} if the completion $\widehat{A}$ is a geometric complete intersection.    Geometric complete intersections are abstract complete intersections but the converse is \emph{not} true, see  \cite[section 2]{HJ}.

Our first result is:
\begin{theorem}\label{theorem-ci}
Let $A, B$ be Gorenstein local rings. Assume $\CMS(A)$ is triangle equivalent to $\CMS(B)$. If $A$ is an abstract complete intersection of co-dimension $c$ then $B$ is also an abstract complete intersection of co-dimension $c$.
\end{theorem}

Let $(A,\m)$ be an  abstract complete intersection of co-dimension $c$. For $1 \leq r \leq c$ let $\CMS_{\leq r}(A)$ be the thick subcategory of $\CMS(A)$ generated by MCM $A$-modules with complexity $\leq r$ (for definition of complexity see \ref{bella-rangella}).
If $A$ and $B$ are abstract complete intersections with infinite residue fields then it essentially follows from techniques used to prove \ref{theorem-ci} that if $\CMS_{\leq r}(A)$ is triangle equivalent to $\CMS_{\leq s}(B)$ then necessarily $r = s$. However the following result is a bit un-expected:
\begin{theorem}\label{cm-r}
Let $(A,\m)$, $(B,\n)$ be geometric complete intersections of \\ co-dimension $m,n \geq 2$ respectively and with algebraically closed residue fields. Let  \\
$r \geq \max\{ m/2, n/2 \}$ and also $r \leq \min\{ m, n \}$. If $\CMS_{\leq r}(A)$ is triangle equivalent to $\CMS_{\leq r}(B)$ then necessarily
 $m = n$.
\end{theorem}
Proof of Theorem \ref{cm-r} uses techniques developed by Avramov and Buchweitz in \cite{avr-b}.
We make the following:
\begin{conjecture}
Let $(A,\m)$, $(B,\n)$ be geometric complete intersections of \\ co-dimension $m,n$ respectively and with algebraically closed residue fields. Let  \\
$r \leq \min\{ m, n \}$. If $\CMS_{\leq r}(A)$ is triangle equivalent to $\CMS_{\leq r}(B)$ then necessarily
 $m = n$.
\end{conjecture}

\textbf{I (b):}

Our next result shows that in many cases stable equivalence preserves dimension.
\begin{theorem}\label{dim}
Let $(A,\m), (B, \n)$ be Henselian Gorenstein  with $\CMS(A)$ triangle equivalent to $\CMS(B)$.
Then
\begin{enumerate}[\rm (1)]
\item
If $A$ is not a hypersurface ring then $\dim A = \dim B$.
\item
If $A$ is a hypersurface ring having an MCM $A$-module $M$ with $\Omega(M) \ncong M$ then $\dim A - \dim B$ is even.
\end{enumerate}
\end{theorem}
\begin{remark}
\begin{enumerate}
  \item If $A$ is a geometric hypersurface ring with multiplicity $\geq 3$ then there exists an MCM $A$-module $M$ with $\Omega(M) \ncong M$.
 (\emph{Sketch of a proof:}) A geometric hypersurface  has an Ulrich module  $U$ (i.e., $U$ is an MCM $A$-module with multiplicity equal to its number of minimal generators), see \cite[2.4]{HUB}. If multiplicity of $A$ is $\geq 3$ then by  \cite[Theorem 2]{P2}  it follows that
if $U$ is any Ulrich $A$-module then $\Omega(U) \ncong U$.

\item
Kn\"{o}rrer periodicity \cite{K} gives examples of hypersurface singularities with $\CMS(A)$ triangle equivalent to $\CMS(B)$ and $\dim A - \dim B$ a non-zero
even number. Note that Kn\"{o}rrer does not show the functor defined in \cite[Section 3]{K} is triangulated. However this is not difficult to prove. We also note that Kn\"{o}rrer periodicity does not preserve multiplicity. So multiplicity is not a stable invariant.

 \item
 If $(A,\m)$ is an excellent  Henselian Gorenstein ring which is an isolated singularity  then the functor $ -\otimes \widehat{A} \colon \CMS(A) \rt \CMS(\widehat{A})$ is an equivalence of triangulated categories. This fact is essentially contained in proof of
  Theorem 2.9 in \cite{W}
 (also see \cite[Remark A.6]{KMV}).

 \item
Let $R,S$ be complete \CM \ but not necessarily Gorenstein  and if there is an triangle isomorphism between the singularity categories of $R$ and $S$ then it is possible that
$\dim R - \dim S$ is odd, see \cite{Ka}. We believe the essential reason this occurs is that if $(R,\n)$ is non-Gorenstein then $D_{sg}(R)$ might not be Krull-Schmidt. In fact the example in  \cite[1.5]{Ka} is not Krull-Schmidt.
\end{enumerate}
\end{remark}

\s For a local ring $(A,\m)$ let $\Spec^0(A) = \Spec{A} \setminus \{\m\}$ denote the punctured spectrum of $A$.  Let $\CMS^0(A)$ denote the full subcategory of $\CMS(A)$ generated by MCM $A$-modules free on $\Spec^0(A)$. We note that if $E$ is an MCM $A$-module free on $\Spec^0(A)$, all syzygies of $E$ are also free on $\Spec^0(A)$.
Furthermore $\CMS^0(A)$ is a thick subcategory of $\CMS(A)$.
Proof of Theorem \ref{dim} also shows the following result:
\begin{lemma}\label{isolated}
Let $(A,\m), (B, \n)$ be Henselian Gorenstein  with  triangle equivalence $ \Psi \colon \CMS(A) \rt\CMS(B)$.
Then if $M \in \CMS^0(A)$ then $\Psi(M) \in \CMS^0(B)$. In particular if $A$ is an isolated singularity then so is $B$.
\end{lemma}

We  now discuss the technique used to prove Theorem \ref{dim} and Lemma \ref{isolated}. The notion of Auslander-Reiten (AR) triangles was first introduced by Happel in Hom-finite Krull-Schmidt triangulated categories, see \cite[1.4]{Happel}. Later in \cite[I.2.3]{RVan} it is shown that a Hom-finite Krull-Schmidt triangulated category has (right) AR-triangles if and only if  right Serre-functor. Our observation is that the notion of AR-triangle is categorical. Note if $A$ is Henselian then $\CMS(A)$ is Krull-Schmidt. So one can define notion of AR-triangle(ending at $M$)  when $M$ is indecomposable $ N \xrightarrow{f} E \xrightarrow{g} M \xrightarrow{h} \Om^{-1}(N)$  in $\CMS(A)$ (see \ref{defn-AR-triangle}). As our  category $\CMS(A)$ is no-longer Hom-finite it requires a proof that AR-triangles are unique upto isomorphism of triangles in $\CMS(A)$; see \ref{unique-AR}. Finally we prove that there is an AR-triangle ending at $M$ if and only if  $M \in \CMS^0(A)$. Furthermore  $N \cong \Om^{-d + 2}(M)$.
We then show that if $\psi \colon \CMS(A) \rt \CMS(B)$ is a triangle-equivalence and there is an AR-triangle ending in $M$ then $\psi $ induces naturally an AR-triangle ending in $\psi(M)$.
This proves Lemma \ref{isolated}. Finally we prove Theorem \ref{dim} using some techniques from \cite{AR}.

\textbf{II} \emph{Applications:}

(a) \emph{Essential periodicity of complete intersections:} \\
Recall a module $M$ is said to be periodic if $\Omega^r(M) \cong M$ for some $r \geq 1$.
Note if $A$ is \CM \ then a periodic module is necessarily MCM $A$-module.
Every abstract complete intersection  has a  periodic MCM module $M$ with $\Omega^2(M) \cong M$.
We show:
\begin{theorem}\label{obst-tri}
Let $(A,\m)$ be an abstract complete intersection of codimension $c$. Let $\mathcal{C}$ be some triangulated category with shift functor $\Sigma$. If $\phi \colon \CMS(A) \rt \mathcal{C}$ is a non-zero functor of triangulated categories then there exists $X \neq 0$ in $\mathcal{C}$ with $\Sigma^{2s}(X)  \cong X$ for some $s \geq 1$. If residue field of $A$ is infinite then we may choose $s = 1$.
\end{theorem}
\begin{remark}
We note that in Theorem \ref{obst-tri} the functor $\phi$ is \emph{not} necessarily an equivalence. It is simply \emph{non-zero}.
\end{remark}
A trivial corollary of Theorem \ref{obst-tri} is:
\begin{corollary}
\label{zero}
Let $(A,\m)$ be an abstract complete intersection and let $B$ be a Gorenstein local ring with no periodic modules. Then any triangulated functor $ \Psi \colon \CMS(A) \rt \CMS(B)$ is zero.
\end{corollary}

(b) \emph{Periodic complexes with finite length cohomology :}

Let  $\Fb \colon \cdots \rt F^i \xrightarrow{d^i} F^{i+1} \rt \cdots$ be a co-chain complex of finitely generated free $A$-modules. We say $\Fb$ is periodic if $\Fb \cong \Fb(-s)$ for some $s > 0$. We say $\Fb$ is a minimal complex if $d^i(F^i) \subseteq \m F^{i+1}$ for   all $i$. If $M$ is a periodic $A$-module (i.e.,  $\Om^r(M) \cong M$ for some $r \geq 1$) then splicing a minimal resolution of $M$ we can construct a minimal periodic acyclic co-chain complex.  If we drop the assumption on acyclicty then we may ask whether there exist minimal periodic complexes with finite length cohomology.  We prove
\begin{theorem}\label{periodic-complex}
Let $(R,\m)$ be a Noetherian local ring which is a quotient of a regular local ring. Then there exists a minimal periodic (of period $2s$)co-chain complex $\Fb$ of finitely generated free $R$-modules  (i.e., $\Fb \cong \Fb[2s])$) such that $H^i(\Fb)$ has finite length for all $i$. If $R$ is regular or if residue field of $R$ is infinite we can choose $s = 1$. In this case there exists a minimal complex $\Gb$ of finitely generated free modules with $\Gb = \Gb[2]$ and  $H^i(\Gb)$ has finite length for all $i \in \Z$.
\end{theorem}
\begin{remark}
\begin{enumerate}[\rm (a)]
  \item Theorem \ref{periodic-complex} has no content if $\dim R = 0$;  for in that case we may consider the co-chain complex
\[
\cdots \rt R\rt 0 \rt R \rt 0  \rt \cdots
\]
  \item If $\dim R > 0$ and $\Fb$ is two periodic  with finite length cohomology then necessarily $\rank \Fb^{2i} = \rank \Fb^{2i + 1}$ for all $i \in \Z$. (see Lemma \ref{rank}).
\end{enumerate}
\end{remark}

(b (i))\emph{$2$-periodic co-chain complexes  with finite length cohomology over a regular local ring:}\\
Let $(A,\m)$ be a regular local ring of dimension $d \geq 1$.
As MCM modules over $A$ are free we get $\CMS(A) = 0$. By  \cite[4.4.1]{Bu}  it follows that the  homotopy category of exact acyclic co-chain complexes of finitely generated free modules over $A$ is zero. An initial motivation  to prove Theorem \ref{periodic-complex} for us was to prove whether periodic co-chain complexes of finitely generated free modules (with finite length co-homology) exist up to homotopy over regular local rings.  However we make the following:
\begin{conjecture}\label{mine}
Let $(A,\m)$ be a  regular local ring of dimension $d \geq 1$.  Let $\Fb$ be a non-zero minimal 2-periodic co-chain complex $\Fb$ of finitely generated free $A$-modules  such that $H^i(\Fb)$ has finite length for all $i$.
Then
\begin{enumerate}[\rm (1)]
  \item  $\rank \Fb^i \geq d $ for all $i \in \Z$.
  \item There exists a minimal 2-periodic co-chain complex $\mathbb{G}$ of finitely generated free $A$-modules  such that $H^i(\mathbb{G})$ has finite length for all $i$ and $\rank \mathbb{G}^i = d$ for all $i \in \Z$.
\end{enumerate}
\end{conjecture}
Next we give results which support our conjecture.

(a) We prove (1) for $d = 1,2,3, 4$.

(b) We show it for $d = 1, 2,3$.

In general we construct a minimal $2$-periodic co-chain complex $\mathbb{H}$ with $\rank \mathbb{H}^{i} = 2^d$ and $H^i(\mathbb{H})$ has finite length for all $i$.

Here is an overview of the contents of this paper. In section two we introduce notation and discuss a few
preliminary facts that we need. In section three we give a construction which we use often. We also prove Theorem \ref{obst-tri}.
In the next section we prove Theorem \ref{theorem-ci}. We prove Theorem \ref{cm-r}  in section five.  In the next section we prove Theorem \ref{dim}, Lemma \ref{isolated}.  In section seven we prove Theorem \ref{periodic-complex} and give evidence for the validity of Conjecture \ref{mine}.

\section{notation and preliminaries }
In this section we introduce some notation and discuss some preliminaries.
In this paper all rings are Noetherian and all modules considered are finitely generated. We let $\mathbb{N}$ denote the set of non-negative integers. Nothing in this section is a new result.

\s Let $(A,\m)$ be local with residue field $k$ and let $M$ be an $A$-module. By $\ell(M)$ we denote the length of $M$ and by $\mu(M)$ we denote the number of its minimal generators. Set $\beta_i(M) = \ell(\Tor^A_i(k, M)) = \ell(\Ext_A^i(M, k)) $  the $i^{th}$ betti number of $M$.

\s Let $S = \bigoplus_{n \geq 0}S_n$ be a  graded algebra over a ring $S_0$. Let $M$ be a graded $S$-module.  If $m$ is a homogeneous element of $M$ then we set $|m| = $ degree of $m$.

We need the following well-known result.
\begin{proposition}\label{filt-deg2}
  Let $S = K[X_1, \ldots, X_d]$ be a polynomial ring with $\deg X_i = 2$ for all $i$.  Let $M$ be a graded  $S$-module of dimension $r \geq 1$. Then there exists  homogeneous $x$ of degree $2s$ such that $\ker (M(-|2s|) \xrightarrow{x} M)$ has finite length kernel. If $K$ is infinite then we may choose $s = 1$.
\end{proposition}

\s Let $f \colon \N \rt \Z$ be a function. Recall $f$ is said to be of polynomial type if there exists $p(X)\in \Q[X]$ such that $p(n) = f(n)$ for all
$n \gg 0$. Notice $p(X)$ is uniquely determined by $f$ and we write it as $p_f(X)$. We also write $\deg f = \deg p_f(X)$. We make the convention that the zero polynomial has degree $- \infty$.
The following result is easy to prove.
\begin{proposition}\label{poly-type}
Let $f$ be of polynomial type of degree $r$.
\begin{enumerate}[\rm (1)]
  \item For $a \in \Z$ we have $g(n) = f(n) + f(n + a) $ is of polynomial type and $\deg g \leq r$.
  \item For $a,b  \in \Z $ with $a \geq 0$ we have  $h(n) = f(an + b)$ is of polynomial type and  $\deg h \leq r$.
  \item  For $a,b  \in \Z $ with $a \geq 0$ we have $t(n)  = \sum_{i = 0}^{n-1}f(ia + b)$ is of polynomial type and $\deg t(n) \leq r + 1$. \qed
  \item For $r \geq 0$ the function $s(n) = \sum_{i = 0}^{n}i^r$ is of polynomial type of degree $r + 1$.
\end{enumerate}
\end{proposition}

\s \label{bella-rangella} We need to recall the notion of complexity of a module.  This notion was introduced by Avramov in \cite{LLAV}. More generally we define complexity of a function $f \colon \N \rt \N$.
\[
\cx f = \inf\left \lbrace b \in \mathbb{N}  \left|  \limsup_{n \rt \infty} \frac{f(n))}{n^{b-1}} <  \infty
                                                                 \right. \right \rbrace.
\]
Let $\beta_i^A(M) = \ell( \Tor^A_i(M,k) )$ be the $i^{th}$ Betti number of $M$ over $A$. Let $\beta_M$ be the function defined as $\beta_M(n) = \beta_i^A(M)$.  The complexity of $M$ over $A$ is defined by
\[
    \cx_A M = \cx \beta_M.
\]
Note that $\cx_A M = 0$ if and only if $\projdim_A M < \infty$. Furthermore $\cx_A M  \leq 1$ if and only if $M$ has bounded Betti numbers.
If $A$ is a local abstract
 complete intersection of $\codim c$ then $\cx_A M \leq c$ for any $A$-module $M$; see \cite[4.1]{Gull}.

 \begin{remark}
 \label{equi-complexity}
 We note that $\limsup_{n \rt \infty} f(n)/n^{b-1} <  \infty$ if and only if there exists a polynomial $p(X) \in \mathbb{Q}[X]$ with $\deg p(X) = b -1$ and positive leading coefficient such that $f(n) \leq p(n)$ for all $n \gg 0$.
 \end{remark}
\s\label{max-cx} For $i = 1, \ldots, r$  let $f_i \colon \N  \rt \Z$ be functions.
Set $f(n) = \max\{ f_i(n) \mid 1\leq i \leq r \}$. Then it is clear that
\[
\cx f \leq \max \{ \cx f_i \mid 1\leq i \leq r \}.
\]
We will need the following result. Although elementary we give a proof for the convenience of the reader.
\begin{lemma}\label{bella}
Let $f,g \colon \N \rt \Z$ be functions such that there exists $n_0$ and $d$ such that
\[
f(n + d) \leq f(n) + g(n) + g(n-1) \quad \text{for all} \ n \geq n_0.
\]
If $\cx g \leq r$ then $\cx f \leq r + 1$.
\end{lemma}
\begin{proof}
Set $h(n) = g(n) + g(n -1)$ for $n \geq 1$ and $h(0) = 0$. Then clearly $\cx h \leq r$.
We have $f(n + s) \leq  f(n) + h(n)$ for $n \geq  n_0$.
For $i = 0, \ldots, d -1$  set
\[
f_i(m) = f(n_0 + i + dm)
\]
We have
\[
f_i(m) \leq  f(n_0 + i) +\sum_{ k = 0}^{m-1} h(n_0 + i + kj).
\]
Using \ref{poly-type}(4) we get $\cx f_i \leq r + 1$.
It follows that $\cx f \leq r + 1$.
\end{proof}

 \s \label{small} Let $N$ be an  $A$-module. For each non-negative integer $n$, let $\pi^n$ denote
  the canonical surjection $N\to N/\m^{n+1}N$.
  We say $\pi_n$ is \textit{small} if the induced maps
  $$\Tor_j^A(\pi_n, k) \colon \Tor_j^A(N, k) \rt  \Tor_j^A(N/\m^{n+1}N, k);$$
  is injective for all $j \geq 0$.
  By \cite[(A.4)]{Av1} there exists $n_0$ such that $\pi_n$ is small for all $n \geq n_0$.

\s\label{T-struc} \emph{Triangulated category structure on $\CMS(A)$}. \\
The reference for this topic is \cite[4.7]{Bu}. We first describe a
basic exact triangle. Let $f \colon M \rt N$ be a morphism in $\CMa(A)$. Note we have an exact sequence $0 \rt M \xrightarrow{i} Q \rt \Om^{-1}(M) \rt 0$, with $Q$-free. Let $C(f)$ be the pushout of $f$ and $i$. Thus we have a commutative diagram with exact rows
\[
  \xymatrix
{
 0
 \ar@{->}[r]
  & M
\ar@{->}[r]^{i}
\ar@{->}[d]^{f}
 & Q
\ar@{->}[r]^{p}
\ar@{->}[d]
& \Om^{-1}(M)
\ar@{->}[r]
\ar@{->}[d]^{j}
&0
\\
 0
 \ar@{->}[r]
  &N
\ar@{->}[r]^{i^\prime}
 & C(f)
\ar@{->}[r]^{p^\prime}
& \Om^{-1}(M)
    \ar@{->}[r]
    &0
\
 }
\]
Here $j$ is the identity map on $\Om^{-1}(M)$.
As $N, \Om^{-1}(M) \in \CMa(A)$  it follows that $C(f) \in \CMa(A)$.
Then the projection of the sequence
$$ M\xrightarrow{f} N \xrightarrow{i^\prime} C(f) \xrightarrow{-p^\prime} \Omega^{-1}(M)$$
in $\CMS(A)$  is a basic exact triangle. Exact triangles in $\CMS(A)$ are triangles isomorphic to a basic exact triangle.

\s \label{ext-triangles:} The following assertions are well-known: We note that the lower exact sequence in the above commutative diagram can be considered as an element in $\Ext^1_A(\Om^{-1}(M), N)$. Conversely  consider
an element $s \colon 0 \rt N \rt L \rt \Om^{-1}(M) \rt 0$ of $\Ext^1_A(\Om^{-1}(M), N)$.   Consider the exact sequence $ 0 \rt M \rt Q \rt \Om^{-1}(M) \rt 0$ with
$Q$-free. Let $f \colon M \rt N$ be a lift of identity map on $\Om^{-1}(M)$. Then a simple diagram chase (for instance see  \cite[Chapter XIV, proof of Theorem 1.1]{CE}) shows that connecting surjection $\gamma \colon \Hom_A(M, N) \rt \Ext^1_A(\Om^{-1}(M), N)$ maps $f$ to $s$.
Thus given an element $s \in \Ext^A_1(\Om^{-1}(M), N)$
\[
s \colon 0 \rt N \xrightarrow{\alpha} L \xrightarrow{\beta} \Om^{-1}(M) \rt 0;
\]
there exists a basic exact triangle in $\CMS(A)$
\[
M \xrightarrow{f}  N \xrightarrow{\alpha} L \xrightarrow{-\beta} \Om^{-1}(M)
\]
where $\gamma$ maps $f$ to $s$.
\s It can be easily seen that the surjective map \\ $\gamma \colon \Hom_A(M, N) \rt \Ext^1_A(\Om^{-1}(M), N)$ discussed above induces an isomorphism
\[
\underline{\gamma} \colon \sHom_A(M, N) \rt \Ext^1_A(\Om^{-1}(M), N).
\]

The following result is definitely  known. We give a proof due to lack of a reference.
\begin{lemma}\label{betti-stable}
Let $(A,\m)$ be a Gorenstein local ring and let $M, N, L$ be MCM $A$-modules.
\begin{enumerate}[\rm (1)]
\item
A morphism $f \colon M \rt N$ in $\CMS(A)$ yields for $i \geq 1$ well-defined maps
$$T_i(f) \colon \Tor^A_i(M, k) \rt \Tor^A_i(N, k).$$
\item
If $f \colon M \rt N$ is isomorphism in $\CMS(A)$  then $T_i(f)$ are isomorphisms for $i \geq 1$.
\item
If $M \rt N \rt L \rt \Omega^{-1}M$ is an exact triangle in
$\CMS(A)$ then for $i \geq 2$ we have an exact sequence
\[
\Tor^A_i(M, k) \rt \Tor^A_i(N, k)  \rt \Tor^A_i(N, k) \xrightarrow{\delta_i} \Tor^A_{i-1}(M, k)
\]
\end{enumerate}
\end{lemma}
\begin{proof}
(1) This follows from the fact that if $g \colon M \rt N$ factors through a free module then $\Tor^A_i(g, k) = 0$ for $i > 0$.

(2) This follows from (1).

(3) As rotations of exact triangles are exact,  we get that
there is an exact triangle $\Omega(L) \rt M \rt N \rt L$. By construction of exact triangles in $\CMS(A)$ we may assume that there exists
an exact sequence  $0 \rt M \rt N\oplus F \rt L \rt 0$ in $\text{mod}(A)$ (where $F$ is a finitely generated free
$A$-module). So we have a long exact sequence in homology
\[
\Tor^A_i(M, k)  \rt \Tor^A_i(N, k)  \rt \Tor^A_i(N, k) \xrightarrow{\delta_i} \Tor^A_{i-1}(M, k)
\]
\end{proof}

\section{section of a module over complete intersection}
Let $(A,\m)$ be an  abstract complete intersection of co-dimension $c$. For $1 \leq r \leq c$ let $\CMS_{\leq r}(A)$ be the thick subcategory of $\CMS(A)$ generated by MCM $A$-modules with complexity $\leq r$. If $c \geq 2$ then  for $r \geq 2$ we let $\tau_r(A) = \CMS_{\leq r}(A)/\CMS_{\leq r-1}(A)$ be the
Verdier quotient.
For systemic reasons we set $\tau_1(A) = \CMS_{\leq 1}(A)$.
We say an object in a $X$ triangulated category
$\mathcal{C}$ with shift functor $\Sigma$ is periodic  if $\Sigma^{s}(X) = X$ for some $s \geq 1$. We show
\begin{theorem}\label{periodic}
Let $(A,\m)$ be an abstract  complete intersection of co-dimension $c$. Fix $i$ with $i = 1, \ldots, c$ and let $X \in \tau_i(A)$. Then $X$ is periodic with period $ =  2s $, for some $s \geq 1$ (here $s$ may depend on $X$). If the residue field of $A$ is infinite then we can choose $s = 1$.
\end{theorem}

\s To prove Theorem \ref{periodic} we need the notion of cohomological operators over a complete intersection ring; see \cite{Gull} and
\cite{Eis}.
 Let $\mathbf{f} = f_1,\ldots,f_c$ be a regular sequence in a  local Noetherian ring $Q$. Set $I = (\mathbf{f})$ and
 $ A = Q/I$,
\s

The \emph{Eisenbud operators}, \cite{Eis}  are constructed as follows: \\
Let $\mathbb{F} \colon \cdots \rightarrow F_{i+2} \xrightarrow{\partial} F_{i+1} \xrightarrow{\partial} F_i \rightarrow \cdots$ be a complex of free
$A$-modules.

\emph{Step 1:} Choose a sequence of free $Q$-modules $\wt{F}_i$ and maps $\wt{\partial}$ between them:
\[
\wt{\mathbb{F}} \colon \cdots \rightarrow \wt{F}_{i+2} \xrightarrow{\wt{\partial}} \wt{F}_{i+1} \xrightarrow{\wt{\partial}} \wt{F}_i \rightarrow \cdots
\]
so that $\mathbb{F} = A\otimes\wt{\mathbb{F}}$

\emph{Step 2:} Since $\wt{\partial}^2 \equiv 0 \ \text{modulo} \ (\mathbf{f})$, we may write  $\wt{\partial}^2  = \sum_{j= 1}^{c} f_j\wt{t}_j$ where
$\wt{t_j} \colon \wt{F}_i \rightarrow \wt{F}_{i-2}$ are linear maps for every $i$.

 \emph{Step 3:}
Define, for $j = 1,\ldots,c$ the map $t_j = t_j(Q, \mathbf{f},\mathbb{F}) \colon \mathbb{F} \rightarrow \mathbb{F}(-2)$ by $t_j = A\otimes\wt{t}_j$.

\s
The operators $t_1,\ldots,t_c$ are called Eisenbud's operator's (associated to $\mathbf{f}$) .  It can be shown that
\begin{enumerate}
\item
$t_i$ are uniquely determined up to homotopy.
\item
$t_i, t_j$ commute up to homotopy.
\end{enumerate}
\s Let $R = A[t_1,\ldots,t_c]$ be a polynomial ring over $A$ with variables $t_1,\ldots,t_c$ of degree $2$. Let $M, N$ be  finitely generated $A$-modules. By considering a free resolution $\mathbb{F}$ of $M$ we get well defined maps
\[
t_j \colon \Ext^{n}_{A}(M,N) \rightarrow \Ext^{n+2}_{A}(M,N) \quad \ \text{for} \ 1 \leq j \leq c  \ \text{and all} \  n,
\]
which turn $\Ext_A^*(M,N) = \bigoplus_{i \geq 0} \Ext^i_A(M,N)$ into a module over $R$. Furthermore these structure depend  on $\Bf$, are natural in both module arguments and commute with the connecting maps induced by short exact sequences.

\s  Gulliksen, \cite[3.1]{Gull},  proved that if $\projdim_Q M$ is finite then
$\Ext_A^*(M,N) $ is a finitely generated $R$-module. By taking $N = k$ we get that the function $n \rt \beta_n(M)$ is quasi-polynomial of period $2$ and degree $\cx M - 1$.

The following result is crucial to prove some of our results. When $A$ is a geometric complete intersection with infinite residue field then Theorem \ref{slice}  essentially
follows from techniques in \cite[7.3]{avr-gp}.
\begin{theorem}\label{slice}
Let $(A,\m)$ be a complete intersection  with residue field $k$ and let $M$ be a MCM $A$-module of complexity $\geq 2$. Set $M_i = \Omega^i(M)$ for
$i \geq 0$.  Then there exists $s$ and $n_0$ such that we have a surjective map $M_{n_0 + 2s} \xrightarrow{\alpha} M_{n_0}$ satisfying the following  properties:
\begin{enumerate}[\rm (1)]
\item
The maps $\Tor^A_i(\alpha, k) $ is surjective for all $i \geq 0$.
\item
$\cx \ker \alpha = \cx M   - 1$.
\item
If $k$ is infinite we can choose $s = 1$.
\end{enumerate}
\end{theorem}
\begin{proof}
We first consider the case when $A$ is complete. Then $A = Q/(\Bf)$ where $(Q, \n)$ is a complete regular local ring and $\Bf = f_1,\ldots, f_c \in \n^2$ is a $Q$-regular sequence. Let $\Fb$ be a minimal resolution of $M$ as an $A$-module and let $t_1,\ldots,t_c \colon \Fb(+2) \rt \Fb $ be Eisenbud operators corresponding to $\Bf$ and $\Fb$. Set $E(M) = \bigoplus_{n \geq 0}\Ext^n_A(M, k)$. It is a finitely generated graded $R = A[t_1,\ldots, t_c]$-module where $\deg t_i = 2$ for all $i$.
As $\m E(M) = 0$ we get that $E(M)$ is a finitely generated $S = k[t_1,\ldots, t_c]$-module. We may choose $t$ homogeneous of degree $2s $ in $S$ such that $(0 \colon_{E(M)} t)$ has finite length (equivalently $t$ is $E(M)$-filter regular). If $k$ is infinite it is readily seen that we may choose $t$ linear in $t_i$ and so of degree $2$ (see \ref{filt-deg2}).
We have that $t$ induces injections $\Ext^{n}_A(M, k) \rt \Ext^{n+ 2s}_A(M, k)$ for $n \gg 0$ say for $n \geq n_0$. Dualizing we get surjections $\Tor^A_{n+2s}(M, k) \rt \Tor^A_{n}(M, k)$ for $n \geq n_0$. Let $\xi$ be homogeneous of degree $2s$ in $R$ such that its image in $S$ is $t$. So we have a chain map $\xi \colon \Fb(-2s) \rt \Fb$. As $\Tor^A_{n}(\xi, k)$ is surjective for $n \geq n_0$, it follows from Nakayama's lemma  that we have surjections $\Fb_{n + 2s} \rt \Fb_n$ for $n \geq n_0$. Note $\xi$ is a chain map. So we have exact sequences $ 0 \rt L_n \rt M_{n+ 2s} \rt M_n \rt 0$ for $n \geq n_0$. Also note that we also have exact sequences  for $i \geq 0$
\[
\Tor^A_{i}(M_{n_0 + 2s}, k) \rt \Tor^A_{i}(M_{n_0}, k) \rt 0.
\]
 It follows that we have exact sequence
\[
 0 \rt \Tor^A_{i}(L_{n_0}, k) \rt \Tor^A_{i}(M_{n_0 + 2s}, k) \rt \Tor^A_{i}(M_{n_0}, k) \rt 0.
\]
So $\cx L_{n_0} = \cx M - 1$.
The result follows if $A$ is complete.

We now consider the general case. Note $\Omega_i^A(M)\otimes \widehat{A} \cong \Omega_i^{\widehat A}(\widehat{M})$.
By the complete case we have an exact sequence
\[
0 \rt L \rt \widehat{M}_{n_0 + 2s} \xrightarrow{\alpha}  \widehat{M}_{n_0} \rt 0
\]
which does the job.
Notice
\[
\alpha \in \Hom_{\widehat{A}}(\widehat{M}_{n_0 + 2s}, \widehat{M}_{n_0}) = \Hom_A(M_{n_0 + 2s}, M_{n_0})\otimes_A \widehat{A}.
\]
Set $\alpha = \sum_{i = 1}^{r} \phi_i \otimes a_i$ where
$\phi_i \in \Hom_A(M_{n_0 + 2s}, M_{n_0})$ and $a_i \in \widehat{A}$.

Now assume that the map $\widehat{M_{n_0}} \rt \widehat{M_{n_0}}/\widehat{\m}^l \widehat{M_{n_0}}$ is small (see \ref{small}).
Let $a_i = b_i + c_i$ with $b_i \in A$ and $ c_i \in \widehat{\m}^l$.
So we get $\alpha = \widetilde{\alpha}\otimes_A 1 + \beta$ where
 $\widetilde{\alpha} \in \Hom_A(M_{n_0 + 2s}, M_{n_0})$ and
 $\beta \in \widehat{\m}^l\Hom_{\widehat{A}}(\widehat{M}_{n_0 + 2s}, \widehat{M}_{n_0})$
It follows that
\[
\Tor^{\widehat{A}}_*(\alpha, k) = \Tor^{A}_*(\alpha, k)\otimes_A\widehat{A}  + \Tor^{\widehat{A}}_*(\beta, k)
\]
We have a commutative diagram
\[
\xymatrix{
\widehat{M_{n_0 + 2s}}
\ar@{->}[d]_{\beta}
\ar@{->}[dr]^{0}
 \\
 \widehat{M_{n_0}}
\ar@{->}[r]_{\pi}
& \widehat{M_{n_0}}/\widehat{\m}^l \widehat{M_{n_0}}
}
\]
This yields a commutative diagram
\[
\xymatrix{
\Tor^{\widehat{A}}_*(\widehat{M_{n_0 + 2s}}, k)
\ar@{->}[d]_{\Tor^{\widehat{A}}_*(\beta, k)}
\ar@{->}[dr]^{0}
 \\
 \Tor^{\widehat{A}}_*(\widehat{M_{n_0}}, k)
\ar@{->}[r]_{\Tor^{\widehat{A}}_*(\pi, k)}
& \Tor^{\widehat{A}}_*(\widehat{M_{n_0}}/\widehat{\m}^l \widehat{M_{n_0}}, k)
}
\]
As $\Tor^{\widehat{A}}_*(\pi, k)$ is injective it follows that
 $ \Tor^{\widehat{A}}_*(\beta, k) = 0$.
 So we have
 $$ \Tor^{\widehat{A}}_*(\widetilde{\alpha} \otimes 1, k)  =
 \Tor^{\widehat{A}}_*(\alpha, k)$$
 It follows that $ \Tor^{A}_*(\widetilde{\alpha} , k)$ is surjective. Set
 $K = \ker\widetilde{\alpha}$. We have a short exact sequence
 $$ 0 \rt K \rt M_{n_0 + 2s} \rt M_{n_0} \rt 0$$
 which satisfies our requirements.
\end{proof}

We now give:
\begin{proof}[Proof of Theorem \ref{periodic}]
It is well-known that if $M$ is MCM with complexity one then it is periodic with period $2$. So $\tau_1(A)$ is periodic.
Now let $r \geq 2$ and let $M \in \tau_r(A)$ be non-zero. Then $M$ as an $A$-module has complexity $= r$. Set $M_n = \Omega_n(M)$. Then by Theorem \ref{slice} there exists an exact sequence $0 \rt K \rt M_{n+2s} \rt M_n \rt 0$ with complexity $K = r-1$. So in $\CMS(A)$ we have an exact triangle
$K \rt M_{n+2s} \rt M_n \rt \Omega^{-1}{K}$. It follows that in $\tau_r(A)$ we have $M_{n + 2s} \cong M_n$ and so $M_{2s} \cong M$. So $M$ is periodic with period $2s$. Also by our construction in Theorem \ref{slice} we may choose $s = 1$ if the residue field of $A$ is infinite.
\end{proof}

We now give a proof of Theorem \ref{obst-tri}. It is convenient to prove a slightly more general result.
\begin{theorem}\label{ascend-periodic}
Let $\mathcal{C}$ be a triangulated category. Let $\mathcal{C}_1 \subseteq \mathcal{C}_2  \subseteq
\ldots$ be an ascending sequence of thick triangulated subcategories of $\mathcal{C}$ such that
\begin{enumerate}[\rm (1)]
  \item $\mathcal{C} = \bigcup_{i \geq 1} \mathcal{C}_i$.
  \item $\mathcal{C}_1$ is periodic.
  \item The Verdier quotient $\mathcal{C}_i/\mathcal{C}_{i-1}$ is periodic for all $i \geq 2$.
\end{enumerate}
If $f \colon \mathcal{C} \rt \mathcal{H}$ is a non-zero triangulated functor then there exists a non-zero
periodic object  $X$ in $\mathcal{H}$. Furthermore if $\mathcal{C}_1$ and $\mathcal{C}_i/\mathcal{C}_{i-1}$ are $2$-periodic for $i \geq 2$ then we may choose $X$ to be two periodic.
\end{theorem}
\begin{proof}
We note that either $f(\mathcal{C}_1) \neq 0$ or there exists $r \geq 2$ such that $f(\mathcal{C}_{r}) \neq 0$ but $f(\mathcal{C}_{r-1}) = 0$,
If $f(\mathcal{C}_1) \neq 0$ then if $X$ is a non-zero element in the image we get that $X$ is periodic.

If there exists $r \geq 2$ such that $f(\mathcal{C}_{r}) \neq 0$ but $f(\mathcal{C}_{r-1}) = 0$ then note that $f $ factors over the Verdier quotient $\ov{f} \colon \mathcal{C}_r/\mathcal{C}_{r-1} \rt H$ and is non-zero. It follows that if $\ov{f}(M) = X$ is non-zero then $X$ is periodic.
\end{proof}

We now give
\begin{proof}[Proof of Theorem \ref{obst-tri}]
The result follows from Theorem \ref{ascend-periodic} and Theorem \ref{periodic}.
\end{proof}

\s Let $(A,\m)$ be an abstract complete intersection of codimension $c$. Then $\CMS^0(A)$ the category of MCM $A$-modules free on $\Spec^0(A)$ is a thick subcategory of $\CMS(A)$.
 Set
$$ \CMS^0_{\leq i}(A) = \{ M \mid M\in \CMS^0(A) \ \text{and} \ \cx M \leq i \}. $$
Then $\CMS^0_{\leq i}(A)$ defines a thick subcategory of $\CMS(A)$. For $i = 2,\cdots, c$ set $V_i(A) = \CMS^0_{\leq i}(A)/\CMS^0_{\leq {i-1}}(A)$. For systemic reasons set
$V_1(A) = \CMS^0_{\leq 1}(A)$.
\begin{remark}
The categories $V_i(A)$ are non-zero for $i = 1,\ldots c$. It suffices to show there exist MCM $A$-modules  in  $\CMS^0(A)$ of complexity $i$ for $i = 1, \cdots, c$. We note that if $d = \dim A$ and $k$ is the residue field of $A$ then $\Om^d(k)$ is MCM $A$-module of complexity $c$. Using Theorem \ref{slice} iteratively our assertion follows.
\end{remark}
We prove:
\begin{theorem}\label{vb1}
(with hypotheses as above:) The triangulated categories d $V_i(A)$ for $i = 1,\cdots, c$ are periodic. If the residue field of $A$ is infinite then they are $2$-periodic.
\end{theorem}
\begin{proof}
  Any MCM $A$-module with $\cx M \leq 1$ is periodic with period $2$. So $V_1(A)$ is periodic.

  Let $M \in V_i(A)$ be non-zero for some $i \geq 2$. Then $M$ as an $A$-module has complexity $i$. Set $M_n = \Omega^n(M)$. Then by Theorem \ref{slice} we get an MCM $A$-module $K$ of complexity $i-1$ and an exact sequence $0 \rt K \rt M_{n+2s} \rt M_n \rt 0$ in $A$ for some $s >0$ and for some $n \geq 0$. Note $(M_i)_P $ is free for all $P \neq \m$ and for all $i \geq 0$. It follows that $K_P$ is free for all $P \neq \m$. It follows that $K \in \CMS^0_{\leq i-1}(A)$. We also have an exact triangle $K \rt M_{n+2s} \rt M_n \rt \Omega^{-1}(K)$ in $\CMS^0_{\leq i}(A)$. So in $V_i(A)$ we get $M_{n+ 2s} \cong M_n$ and so $M_s \cong M$. Thus $V_i(A)$ for $i \geq 2$ are periodic. By \ref{slice} if the residue field of $A$ is infinite then we may choose $s = 1$. So in this case $V_i(A)$ for $i \geq 1$ are $2$-periodic.
\end{proof}

As an easy consequence we get
\begin{corollary}\label{vb-cor}
(with hypotheses as above:) Let $\mathcal{C}$ be a triangulated category and  let $f \colon \CMS^0(A) \rt \mathcal{C}$ be a non-zero triangulated functor.
Then there exists $0 \neq X \in \mathcal{C}$ which is periodic. If the residue field of $A$ is infinite then we can choose $X$ to be $2$-periodic.
\end{corollary}

\begin{proof}
  The result follows from Theorem \ref{ascend-periodic} and Theorem \ref{vb1}.
\end{proof}
\section{Proof of Theorem \ref{theorem-ci}}
In this section we first prove the following:
\begin{theorem}\label{basic}
  Let $(A,\m)$ be an abstract complete intersection of codimension $c$. Let $(B,\n)$ be a Gorenstein local ring.
  Assume there exists a triangulated functor $f \colon \CMS(A) \rt \CMS(B)$. If $M$ is an MCM $A$-module with $\cx_A M = r$ then $\cx_B f(M) \leq r$.
\end{theorem}
\begin{proof}
We note that every $A$ module $N$ has complexity $\leq c$.
We prove the result by induction on $\cx M$. If $\cx M = 0$ then $M$ is a free $A$-module. So $M = 0$ in $\CMS(A)$. Thus $f(M) = 0$ in $\CMS(B)$. So $f(M)$ is a free $B$-module.

Next we assume $\cx_A M = 1$. Then $M$ is periodic. We may assume $M$ has no free summands. Then $\Omega_A^2 (M) \cong M$.
As $f$ is a triangulated functor we get that $\Om_B^2 f(M) \cong f(M)$. So $\cx_B f(M) \leq 1$

Now assume that $\cx_A M  = r \geq 2$ and that our claim is true for MCM $A$-modules with complexity $= r -1$.

 Set $M_n = \Omega^n_A(M)$. By Theorem \ref{slice} there exists $n_0$ and  an exact sequence
\[
0 \rt K \rt M_{n_0 + 2s} \rt M_{n_0} \rt 0
\]
in $\text{mod}(A)$ where $\cx K = r -1$.  Note $K$ is MCM $A$-module. So we have an exact triangle
\[
K \rt M_{n_0 + 2s} \rt M_{n_0} \rt \Om^{-1}_A(K)
\]
So we have an exact triangle in $\CMS(B)$
\[
f(K) \rt f(M)_{n_0 + 2s} \rt f(M)_{n_0} \rt \Om^{-1}_B(f(K)).
\]
By induction hypothesis we get $\cx f(K) \leq r-1$.
Usin \ref{betti-stable}(3) we have for all $n \geq n_0 + 2$
\[
\beta^B_{n + 2s}(f(M)) \leq \beta^B_{n }(f(M)) + \beta^B_{n - n_0}(f(K)) + \beta^B_{n - n_0  - 1}(f(K))
\]
 By Lemma \ref{bella}  it follows that $\cx f(M) \leq r$.
\end{proof}

Recall an additive functor $f \colon \CMS(A) \rt \CMS(B)$ is a weak equivalence if every MCM $B$-module $N$ is a direct summand of $f(M_N)$ for some $M_N$ in $\CMS(A)$ (here $M_N$ depends on $N$).

\begin{corollary}\label{weak-equivalence}
Let $f \colon \CMS(A) \rt \CMS(B)$ be a weak equivalence. If $A$ is an abstract complete intersection of codimension $c$ then $B$ is also an abstract complete intersection of codimension $\leq c$.
\end{corollary}
\begin{proof}
  Let $N$ be an MCM $B$-module. Say $N$ is a direct summand of $f(M)$. As $\cx_A M \leq c$, by \ref{basic} we get that $\cx_B f(M) \leq c$.
  So $\cx_B N \leq c$. By Gulliksen's result, \cite[2.3]{Gull-dev} we have that $B$ is also an abstract complete intersection of codimension $\leq c$.
\end{proof}

Next we give a proof of Theorem \ref{theorem-ci}. We restate it for the convenience of the reader.
\begin{theorem}\label{theorem-ci-body}
Let $A, B$ be Gorenstein local rings. Assume $\CMS(A)$ is triangle equivalent to $\CMS(B)$. If $A$ is an abstract complete intersection of co-dimension $c$ then $B$ is also an abstract complete intersection of co-dimension $c$.
\end{theorem}
\begin{proof}
  Let $f \colon \CMS(A) \rt \CMS(B)$ and $g \colon \CMS(B) \rt \CMS(A)$ be the triangulated functors which are inverses of each other.
  Let $A$ be an abstract complete intersection of codimension $c$. Using $f$ and Corollary \ref{weak-equivalence} we get that $B$ is an abstract complete intersection with codimension $\leq c$. Using $g$ it follows that $c \leq $ codimension of $B$. So $A,B$ have the same codimension.
\end{proof}
We also show the following.
\begin{theorem}\label{theorem-ci-body-2}
Let $A, B$ be abstract complete intersection local rings. Assume $\CMS(A)_{\leq r}$ is triangle equivalent to $\CMS(B)_{\leq s}$ then $r = s$.
\end{theorem}
\begin{proof}
Let $f \colon \CMS(A)_{\leq r} \rt \CMS(B)_{\leq s}$ be an equivalence. Then note if $X$ is in $\CMS(A)_{\leq r}$
 then by \ref{basic} we get that the
complexity of $f(X) \leq r$. So $r < s $ is not possible. Similarly by considering $f^{-1}$ we get that $s < r$ is not possible. So $r = s$.
\end{proof}

\section{Proof of Theorem \ref{cm-r}}
In this section $(Q,\n)$ is a regular local ring with algebraically closed residue field $k$.  Let $A = Q/(f_1,\ldots, f_c)$ where  $f_1,\ldots, f_c \in \n^2$ is a regular sequence. We denote the maximal ideal of $A$ by $\m$.

\s
Let $U, V$ be two  $A$-modules. Define
\[
\cx_A(U, V) = \inf\left \lbrace b \in \mathbb{N}  \left\vert \right.   \limsup_{n \to \infty} \frac{\mu(\Ext^{n}_{A}(U,V))}{n^{b-1}}  < \infty \right \rbrace.
\]

\s  Let $\Ext^*(U,V) = \bigoplus_{n\geq 0} \Ext^{n}_A(U,V)$ be the total ext module of $U$ and $V$. We consider it as a
 module over the ring of cohomological operators $A[t_1,\ldots, t_c]$.
Since $\projdim_Q U$ is finite $\Ext^*(U,V)$ is a finitely generated
$A[t_1,\ldots, t_c]$-module.

\s Let $\eC(U,V) = \Ext^*(U,V)\otimes_A k$. Clearly $\eC(U,V)$ is a finitely generated
  $T = k[t_1,\ldots, t_c]$-module. (Here degree of $t_i$ is 2 for each $i =1,\ldots, c$).
Set
\[
\fA(U,V) = \ann_T \eC(U,V).
\]
Notice that $\fA(U,V)$ is a homogeneous ideal.
\s We consider the affine space $\mathbb{A}^c(k)$.
Let
$$\cV(U,V) =  \cV\left(\fA(U,V)\right) \subseteq \mathbb{A}^c(k).$$
Since $\fA(U, V)$ is graded ideal we get that $\cV(U,V)$ is a cone.
\s\label{lucho-B}  By a result due to Avramov and Buchweitz \cite[2.4]{avr-b} we get that
\[
\dim \cV(U,V) = \cx_A(M,N).
\]
Set $\cV(U) = \cV(U,k)$. Then $\cV(U) = \cV(U,U)$ and
\[
\cV(U,V) = \cV(U)\cap \cV(V).
\]
\s\label{geometry} It is easier to do geometry in projective space. Let $R = k[z_0,\ldots,z_{c-1}]$ be standard graded. If $f(t_1,\ldots, t_c) \in T$ then let $f_R(z_0,\ldots, z_{c-1})$ be the polynomial in $R$ obtained by replacing $t_i$ by $z_{i-1}
$ for  $i = 1,\cdots, c$. Set
\[
\fA_R(U,V) = \{ f_R(z_0,\ldots, z_{c-1}) \mid f \in \fA(U,V) \}.
\]
Then $\fA_R(U,V)$ is a homogeneous ideal in $R$. Set
$$\cV^*(U,V) =  \cV\left(\fA_R(U,V)\right) \subseteq \mathbb{P}^{c-1}(k).$$
Set $\dim \emptyset = -1$. Clearly
\[
\cx_A(M,N) = \dim \cV^*(U,V) + 1.
\]
We also have $\cV^*(U,V) = \cV^*(U) \cap \cV^*(V)$.
We now give a proof of Theorem \ref{cm-r}. We restate it for the convenience of the reader.
\begin{theorem}\label{cm-r-bella}
Let $(A,\m)$, $(B,\n)$ be geometric complete intersections of \\ co-dimension $m,n \geq 2$ respectively and with algebraically closed residue fields. Let  \\
$r \geq \max\{ m/2, n/2 \}$ and also $r \leq \min\{ m, n \}$. If $\CMS_{\leq r}(A)$ is triangle equivalent to $\CMS_{\leq r}(B)$ then necessarily
 $m = n$.
\end{theorem}
\begin{proof}
  Suppose if possible $m \neq n$. Without loss of generality  we may assume $m \geq n + 1$. Let $\psi \colon \CMS_{\leq r}(A) \rt \CMS_{\leq r}(B)$ be an equivalence.
  By using Theorem \ref{basic} to $\psi$ and $\psi^{-1}$ we get that if $\cx_A E = t$ then $\cx_B \psi(E) = t$.
   In $\mathbb{P}^{m-1}(k)$ consider the following two varieties:
  \begin{enumerate}
    \item $X$ defined by $z_0 = z_1 = \cdots = z_{m-r-1} = 0$.
    \item $Y$ defined by $z_{m-r} = z_{m-r +1} = \cdots = z_{m-1} = 0$.
  \end{enumerate}
  Note $\dim X = r - 1$, $\dim Y = m - r - 1$ and $X \cap Y = \emptyset$. By \cite[2.3]{bergh} there exists MCM $A$-modules $M, N$ with $\cV^*(M) = X$ and $\cV^*(N) = Y$.
  Note $\cx_A M = r $ and $\cx_A N = m - r$. As $r \geq m/2$ we get $m - r \leq r$. So $M, N \in \CMS_{\leq r}(A)$. As $\cV^*(M, N) = \emptyset$ we get $\Ext^i_A(M, N) = 0$ for all $i \gg 0$. Furthermore $\sHom_A(M, \Omega^{-i}_A(N)) \cong \Ext^i_A(M, N) = 0 $ for all $i \gg 0$.
  It follows that $\sHom_B(\psi(M), \Omega^{-i}_B(\psi(N)))  = 0$ for all $i  \gg 0$. Thus  $\Ext^i_B(\psi(M),\psi(N)) = 0$ for all $i \gg 0$.
  We now consider support varieties of $B$-modules in $\mathbb{P}^{n-1}$. We have $\cV^*(\psi(M), \psi(N)) = \emptyset$. But
  $\dim \cV^*(\psi(M)) = \cx_A(\psi(M)) - 1 = r -1$. Similarly $\dim \cV^*(\psi(N)) = m -r -1$.  The sum of these two dimensions is $m - 2 \geq n-1$. This implies that
  $\cV^*(\psi(M)) \cap \cV^*(\psi(N)) \neq \emptyset$ cf. \cite[I.7.2]{Hart}.  So we get $\cV^*(\psi(M),\psi(N)) \neq \cV^*(\psi(M)) \cap \cV^*(\psi(N))$ which is a contradiction. So $m = n$.
\end{proof}
\section{ AR-sequences and AR-triangles }
In this section we first assume $(A,\m)$ is a Henselian \CM \ local ring of dimension $d$  with a canonical module $\omega_A$. Later we will restrict to the case when $A$ is a Henselian Gorenstein local ring. We first discuss the notion  of  AR-sequences. A good reference for these concepts is \cite{Y}.
 Finally we
 discuss existence of  AR-triangles in $\CMS(A)$.

\s\label{ar-seq} Recall an exact  sequence $s \colon 0 \rt N \rt E \xrightarrow{p} M \rt 0$ of MCM $A$-modules is an AR-sequence (Auslander-Reiten sequence) if
\begin{enumerate}
 \item $s$ is not split.
 \item $M, N$ are indecomposable maximal \CM \ $A$-modules.
 \item If $L$ is a maximal \CM \ $A$-module and if $q \colon L \rt M$ is a \emph{not} a split epimorphism then there exists
 $f \colon L \rt E$ such that $q = p\circ f$.
\end{enumerate}
We call $s$ the AR-sequence ending at $M$ (equivalently starting at $N$).
If there is a AR-sequence ending at $M$ it is unique up-to isomorphism of short-exact sequences, \cite[2.7]{Y}.

\s Let $M$ be an indecomposable MCM $A$-module. By \cite[3.4]{Y}, there is an AR-sequence ending at $M$ if and only if  $M$ is free on $\Spec^0(A)$.

Also note that $N \cong \Hom_A(\Om^d(\Tr(M), \omega_A)$;
see \cite[3.11]{Y}. We set $\tau(M) = N$ and call it the Auslander-Reiten translate of $M$.

\noindent\emph{From now on assume $(A,\m)$ is a Henselian Gorenstein local ring of dimension $d$.}

\s Let $M$ be an indecomposable MCM $A$-module free on $\Spec^0(A)$. Then the AR-translate $\tau(M) = \Om^{2-d}(M)$.

\s \label{defn-AR-triangle} The notion of AR-triangles in a triangulated category was introduced by Happel \cite[Chapter 1]{Happel}. We will discuss it  only for the category $\CMS(A)$. Note as $A$ is Henselian we get that $\CMS(A)$ is a Krull-Schmidt category.
First note that the notion of split monomorphism and split epimorphism makes sense in any triangulated category. They are called section and retraction respectively in \cite{Happel}.
A triangle $ N \xrightarrow{f} E \xrightarrow{g} M \xrightarrow{h} \Om^{-1}(N)$  in $\CMS(A)$ is called an AR-triangle (ending at $M$) if

   (AR1) \  \ $M, N$ are indecomposable.

   (AR2) \ \  $h \neq 0$.

   (AR3) \  \ If $D$ is indecomposable then for every non-isomorphism $t \colon D \rt M$ we have $h\circ t = 0$.

By considering the functor $\Hom_\T(D, -)$ it is easy to see that (AR3) is equivalent to

  (AR3l)  \  \ If $D$ is indecomposable then for every non-isomorphism $t \colon D \rt M$    there is a lift $q \colon D \rt E$ with $g\circ q = t$.

It is also easy to see that (AR3l) is equivalent to

(AR3g)  \  \ If $W$ is not-necessarily indecomposable and  $s \colon W \rt M$ is not a retraction then    there is a lift $q \colon W \rt E$ with $g\circ q = s$.

 \begin{remark}\label{HuniqueAR}
 It is shown in \cite[1.4.3]{Happel} that AR triangles are unique up-to isomorphism of triangles. However the proof \emph{only works} in Hom-finite categories. In our case the proof works only if $A$ is an isolated singularity. However  to prove some of our results we have to prove uniqueness of AR-triangles in $\CMS(A)$ in general.
 \end{remark}
 We first prove:
 \begin{lemma}\label{end-tri}
 Let $s \colon  N \xrightarrow{f} E \xrightarrow{g} M \xrightarrow{h} \Om^{-1}(N)$  be a triangle in $\CMS(A)$ with $N$ indecomposable and $h \neq 0$. Let
 $\theta \colon s \rt s$ be a morphism of triangles such that $\theta_M = 1_M$.  Then $\theta$ is an isomorphism of triangles.
 \end{lemma}
 \begin{proof}
 We note that $\Om^{-1}(N)$ is indecomposable. If $u = \theta_{\Om^{-1}(N)}$ is not an isomorphism then $u$ is in the Jacobson radical of $ \R = \sHom_A(\Om^{-1}(N), \Om^{-1}(N))$.   As $s$ is a morphism of triangles we get $u \circ h = h$. Therefore $(1_\R- u)\circ h = 0$. Note $1_\R - u$ is invertible. This forces $h = 0$; a contradiction. So $u$ is an isomorphism. It follows that $\theta_N$ is an isomorphism. As $\theta_M, \theta_N$ are isomorphisms it is well known that $\theta_E$ is also an isomorphism. So $\theta$ is an isomorphism of triangles.
 \end{proof}
 We now prove uniqueness of AR-triangles in $\CMS(A)$.
 \begin{proposition}
\label{unique-AR}
In $\CMS(A)$ consider the following two AR-triangles ending at $M$.
\begin{itemize}
  \item $s \colon N \xrightarrow{f} E \xrightarrow{g} M \xrightarrow{h} \Om^{-1}(N)$
  \item $s' \colon  N' \xrightarrow{f'} E' \xrightarrow{g'} M \xrightarrow{h'} \Om^{-1}(N')$
\end{itemize}
Then $s \cong s'$.
 \end{proposition}
 \begin{proof}
   As $h, h'$ are non-zero we get that $g, g'$ are not retractions; see \cite[1.4]{Happel}. By (AR3g) there exists lifts $q \colon E' \rt E$ and $q' \colon E \rt E'$ lifting $g'$ and $g$ respectively. So we have morphism of triangles $\alpha \colon s \rt s'$ and $\beta \colon s' \rt s$ with
   $\alpha_M = \beta_M = 1_M$. Consider $\beta\circ \alpha \colon s \rt s$. As $N$ is indecomposable and $h \neq 0$ we get by Lemma \ref{end-tri} that $\beta \circ \alpha$ is an isomorphism. Similarly $\alpha \circ \beta$ is an isomorphism. So $s \cong s'$.
 \end{proof}

\begin{definition}
  Let $s \colon N \xrightarrow{f} E \xrightarrow{g} M \xrightarrow{h} \Om^{-1}(N)$ be an AR-triangle ending at $M$. Then by Proposition \ref{unique-AR}, $N$ is  determined
  by $M$ (upto a not-necessarily unique  isomorphism). Set $\tau^t(M) = N$.
\end{definition}
 Next we consider the question of existence of AR-triangles ending at an indecomposable MCM $A$-module.
\begin{proposition}
  \label{exist-AR-triangle}Let $M$ be a indecomposable non-free  MCM module.  The following conditions  are equivalent:
  \begin{enumerate}[\rm (1)]
    \item There exists an AR-triangle ending at $M$.
    \item $M$ is free on $\Spec^0(A)$.
  \end{enumerate}
  Furthermore $\tau^t(M) = \Om^{-d + 2}(M)$.
\end{proposition}
\begin{proof}
We first prove (2) $\implies$ (1).

Let
$$s \colon \quad   0 \rt \tau(M) \xrightarrow{\alpha} E \xrightarrow{\beta} M \rt 0$$
 be a AR-sequence ending at $M$. Recall $\tau(M)$ is indecomposable. Consider the natural isomorphism
$\underline{\gamma} \colon \sHom_A(\Om(M), \tau(M)) \rt \Ext^1_A(M, \tau(M))$ and let $\underline{\gamma}(f) = s$. As $s \neq 0$ we get $f \neq 0$.
By \ref{ext-triangles:} there is a triangle in $\CMS(A)$
\[
\Om(M)\xrightarrow{f} \tau(M) \xrightarrow{\alpha} E \xrightarrow{-\beta} M.
\]
 Rotating it we obtain a triangle
\[
\widetilde{s} \colon  \quad \quad \tau(M) \xrightarrow{\alpha} E \xrightarrow{-\beta} M \xrightarrow{-\Om^{-1}(f)} \Om^{-1}(\tau(M)).
\]
We claim that $\widetilde{s}$ is an AR-triangle. Clearly (AR1) is satisfied. Also as $f \neq 0$ we get $\Om^{-1}(f) \neq 0$. So (AR2) is satisfied. Also note that (AR3l)
follows from \ref{ar-seq}(3). So $\widetilde{s}$ is an AR-triangle.

Next we prove (1) $\implies$ (2).\\
Let $s \colon N \xrightarrow{f} E \xrightarrow{g} M \xrightarrow{h} \Om^{-1}(N)$  be the AR-triangle ending at $M$.
Suppose if possible $M$ is not free on $\Spec^0(A)$.  Say $M_P$ is not free for some prime ideal $P \neq \m$ in $A$. Let $\alpha \colon 0 \rt L \rt F \rt M \rt 0 $ be a free cover of $M$. Then notice $\alpha_P \neq 0 $ in $\Ext^1_{A_P}(M_P, L_P) \cong \Ext^1_A(M, L)_P$. So there exists $r \in \m \setminus P$ such that $r^n \alpha \neq 0$ for all $n \geq 1$.
Then note that  the extension $r^n\alpha$ is given by the
push-out diagram:
 \[
  \xymatrix
{
\alpha
\colon
 0
 \ar@{->}[r]
  & L
    \ar@{->}[d]^{r^n}
\ar@{->}[r]^{i}
 & F
    \ar@{->}[d]
\ar@{->}[r]^{\pi}
& M
    \ar@{->}[d]^{1_M}
\ar@{->}[r]
 &0
 \\
 r^n \alpha
 \colon
 0
 \ar@{->}[r]
  & L
\ar@{->}[r]^{i_n}
 & W_n
\ar@{->}[r]^{\pi_n}
& M
\ar@{->}[r]
&0
 }
\]
The above commutative diagram induces a map of triangles
\[
  \xymatrix
{
\beta
\colon
\
  &L
    \ar@{->}[d]^{r^n}
\ar@{->}[r]^{i}
 & F
    \ar@{->}[d]
\ar@{->}[r]^{\pi}
& M
    \ar@{->}[d]^{1_M}
\ar@{->}[r]^{w}
 &\Om^{-1}(L)
 \ar@{->}[d]^{r^n}
 \\
 \beta_n
 \colon
 \
  &L
\ar@{->}[r]^{i_n}
 & W_n
\ar@{->}[r]^{\pi_n}
& M
\ar@{->}[r]^{w_n}
&\Om^{-1}(L)
 }
\]
\textit{Claim:} $\pi_n$ is not a retraction for all $n \geq 1$.

Suppose there exists $m  \geq 1$ such that $\pi_m$ is a retraction. So there exists $\xi \colon M \rt W_m$ such that $\pi_m\circ \xi = 1_M$ in $\CMS(A)$. It follows that in $\md(A)$  the map $1_M - \pi_m\circ \xi \colon M \rt M$ factors through a  free $A$-module. It follows \cite[2.2(2)]{AR} that $\pi_m \circ\xi$ is an isomorphism.  This implies that $r^m \alpha$ is split; which is a contradiction.

As $\pi_n$ is not a retraction and $s$ is an AR-triangle it follows that $\pi_n$ has a lift $W_n \rt E$. So we have a morphism of triangles
$\theta \colon \beta_n \rt s$ with $\theta_M = 1_M$. So we get
\begin{align*}
  h &= \theta_{\Om^{-1}(L)} \circ w_n \\
   &= \theta_{\Om^{-1}(L)} \circ r^nw.
\end{align*}
It follows that $h \in \m^n\sHom_A(M, \Om^{-1}(N)) $ for all $n \geq 1$. By Krull's intersection we get $h = 0$ which contradicts (AR2).

By our proof of (2) $\implies $ (1) we get $\tau^t(M) = \Om^{-d+2}(M)$.
\end{proof}
\begin{lemma}\label{fart}
Let $(A,\m), (B,\n)$ be Henselian Gorenstein rings. Let \\ $\Psi \colon \CMS(A) \rt \CMS(B)$ be an equivalence of triangulated categories with inverse $\Phi$. \\
Let $\eta_X \colon \Psi(\Om_A^{-1}(X)) \rt \Om_B^{-1}(\Psi(X))$ be the natural isomorphism associated to $\Psi$.
We then have:
\begin{enumerate}[\rm (1)]
  \item If $M$ is an indecomposable MCM $A$-module then $\Psi(M)$ is also indecomposable.
  \item Let $$ s: N \xrightarrow{f} E \xrightarrow{g} M \xrightarrow{h} \Om_A^{-1}(N)$$  be  an AR-triangle in $\CMS(A)$.
  Then $$ \Psi(s)\colon  \Psi(N) \xrightarrow{\Psi(f)} \Psi(E) \xrightarrow{\Psi(g)} \Psi(M) \xrightarrow{\eta_N \circ\Psi(h)} \Om_B^{-1}(\Psi(N))$$
is an AR-triangle in $\CMS(B)$.
\end{enumerate}
\end{lemma}
\begin{proof}
  (1) This is clear.

  (2) By (1) we get that $\Psi(M), \Psi(N)$ are indecomposable. Also clearly $\eta_N \circ\Psi(h)\neq 0$. Let $W$ be an indecomposable module in $\CMS(B)$ and let $t \colon W \rt \Psi(M)$ be a map which is not a isomorphism. Then  $\Phi(t) \colon \Phi(W) \rt M$ is a not an isomorphism with $\Phi(W)$ an indecomposable module in $\CMS(A)$. As $s$ is an AR-triangle in $\CMS(A)$ we get that $h \circ \Phi(t) = 0$ It follows that $\eta_N \circ\Psi(h) \circ t =  0$. Thus $\Psi(s)$ is an AR-triangle in $\CMS(B)$.
\end{proof}
As a consequence of the above result we give,
\begin{proof}[Proof of  Lemma \ref{isolated}]
As $\Psi$ is an additive functor, it suffices to show that if $M$ is  indecomposable and free on $\Spec^0(A)$ then $\Psi(M)$ is free on $\Spec^0(B)$.
By \ref{exist-AR-triangle} there exists an AR triangle $s$ in $\CMS(A)$ ending at $M$. By \ref{fart} we get that $\Psi(s)$ is an AR-triangle in $\CMS(B)$ ending at
$\Psi(B)$. So again by \ref{exist-AR-triangle} we get that $\Psi(M)$ is free on $\Spec^0(B)$.
\end{proof}
We now give a proof of Theorem \ref{dim}. We restate it for the convenience of the reader.
\begin{theorem}\label{mom}
Let $(A,\m), (B, \n)$ be Henselian Gorenstein  with $\CMS(A)$ triangle equivalent to $\CMS(B)$.
Then
\begin{enumerate}[\rm (1)]
\item
If $A$ is not an abstract hypersurface ring then $\dim A = \dim B$.
\item
If $A$ is an abstract hypersurface ring having an MCM $A$-module $M$ with $\Omega(M) \ncong M$ then $\dim A - \dim B$ is even.
\end{enumerate}
\end{theorem}

\s \label{P} Let $M$ be an $A$-module. For $i \geq 0$  let $\beta_i(M) = \dim_k \Tor^A_i(M, k)$ be its $i^{th}$ \emph{betti}-number.  Let $P_M(z) = \sum_{n \geq 0}\beta_n(M)z^n$, the \emph{Poincare series} of $M$.
 Set
 \[
 \curv{M} =  \limsup (\beta_n(M))^{\frac{1}{n}}.
 \]
 It is possible that $\cx(M) = \infty$, see \cite[4.2.2]{Av-inf}. However $\curv(M)$ is finite for any module $M$  \cite[4.1.5]{Av-inf}.  It can be shown that if $\cx(M) < \infty $ then $\curv(M) \leq 1$. We also have that $ \curv{M} \leq \curv{k}$, see \cite[4.2.4]{Av-inf}.

\begin{proof}[Proof of Theorem \ref{mom}]
Set $r = \dim A$ and $s = \dim B$. Let $\Psi \colon \CMS(A) \rt \CMS(B)$ be an equivalence with inverse $\Phi$.

  Case 1: $A$ is \emph{not} a hypersurface ring.\\
   Suppose if possible $r \neq s$. Without loss of generality we may assume $r < s$.
  Let $X(\m)$ be the MCM approximation of $\m$. Let  $U$ be an indecomposable summand of $X(\m)$. Then by \cite[4.1]{AR} $U$ is extremal; i.e., if  $A$ is a complete intersection
  of codimension $c \geq 2$ then $\cx_A U = c$ and if $A$ is not a complete intersection then $\curv U = \curv k > 1$. In particular there exists $n_0$ such that
  $\beta^A_{n + 1}(M) > \beta^A_n(M)$ for all $n \geq n_0$. Set $V = \Om^{n_0}_A(U)$.

  Notice $X(\m)$ is free on $\Spec^0(A)$. It follows that $U$ and hence $V$ is free on $\Spec^0(A)$. By \ref{exist-AR-triangle}  we have an AR-triangle ending at $V$ in $\CMS(A)$.
  \[
  \alpha \colon \tau_A^t(V) \rt E \rt V \rt \Om^{-1}_A( \tau_A(V))
  \]
  By \ref{exist-AR-triangle} $\tau_A^t(V) = \Om^{-r + 2}_A(V)$.  By Lemma \ref{fart} $\Psi(\alpha)$ is an AR-triangle ending at $\Psi(V)$.
  So  $\Psi(\tau_A^t(V)) \cong \Om_B^{-s + 2}(\Psi V)$. But we also have
  \[
  \Psi(\tau^t_A(V)) = \Psi(\Om_A^{-r + 2}(V)) \cong \Om_B^{-r + 2}(\Psi(V)).
  \]
  So we have $\Om_B^{s-r}(\Psi(V)) \cong \Psi(V)$. Applying  $\Phi$ we get  $\Om^{s-r}_A(V) \cong V$. This implies $\beta^A_{s-r}(V) = \beta^A_0(V)$ which is a contradiction.

Case 2: $A$ is an abstract hypersurface ring having an MCM $A$-module $M$ with $\Omega(M) \ncong M$ then $\dim A - \dim B$ is even.

By Theorem \ref{theorem-ci}; $B$ is also an abstract hypersurface.
If $r = s$  then  we have nothing to prove. Notice $\Omega_B(\Psi(M)) \ncong \Psi(M)$. So we may without any loss of generality assume $r  < s$.
As in case (1) we get $\Om_A^{s-r}(M) \cong M$. It follows that $s-r$ is even.
\end{proof}

\s \label{ar-triangle-quiver} We now assume $(A,\m)$  is a Henselian Gorenstein isolated singularity and algebraically closed residue field. Let $\Gamma(A)$ be the AR-quiver of $A$ and let $\Gau(A)$ be the stable AR-quiver of $A$.
Recall the vertices of $\Gau(A)$ are isomorphism classes of indecomposable non-free MCM  $A$-modules. Furthermore if $M, N$ are non-free MCM indecomposable $A$-modules.
Let $0 \rt \tau(M) \rt E_M \rt M \rt 0$ be the AR-sequence ending at $M$. Let $r$ be the number of copies of $N$ in direct summands of $E_M$ (note $r = 0$ is possible).
 Then there  $r$ arrows from $N$ to $M$ in $\Gau(A)$; see \cite[5.5]{Y}. Note any AR-triangle ending at $M$ is isomorphic to
\[
s_M \colon \tau(M) \rt E_M \rt M \rt \Om^{-1}(\tau(M).
\]
We call $E_M$ to be the middle term of $s_M$.
We now give
\begin{corollary}
(with hypotheses as in \ref{ar-triangle-quiver}. Suppose $\CMS(A)$  is triangle equivalent to $\CMS(B)$. Then $\Gau(A) \cong \Gau(B)$
\end{corollary}
\begin{proof}
Let $\Psi \colon \CMS(A) \rt \CMS(B)$ be an equivalence with inverse $\Phi$. Let $\Gau(A),\Gau(B)$ be stable AR-quiver of $A$ and $B$ respectively. Define
$f \colon \Gau(A) \rt \Gau(B)$ by mapping $[M] $ to $[\Psi(M)]$. Let $s_M$ be the AR-triangle ending at $M$. If there are $r$ arrows from $N$ to $M$ in $\Gau(A)$, then there are precisely $r$ direct summands of $N$ in the middle term of $s_M$.
By Lemma \ref{fart} we get that that $\Psi(s_M) $ is the AR-triangle ending at $\Psi(M)$. By looking at the middle term of $\Psi(s_M)$ it follows that there are precisely $r$
arrows from $\Psi(N)$ to $\Psi(M)$. It follows that $f$ is an isomorphism of graphs.
\end{proof}

\section{Periodic Complexes}
We first discuss some preliminaries regarding complexes. Let $(R,\m)$ be a Noetherian local ring and let $\md(R)$ be the category of finitely generated $R$-modules.
Let $\Cc(\md(R))$ be the category of  complexes (possibly unbounded at both ends) in $\md(A)$ and let $\Kc(\md(R))$ be the corresponding homotopy category.
Let $\Kc(\proj R)$ be the subcategory of $\Kc(\md(R))$ consisting of complex of projective $A$-modules. Let $\Fb \in \Kc(\proj R)$. Then we say $\Fb$ is minimal if $\partial(\Fb) \subseteq \m \Fb$. We index complexes cohomologically. By a simple split exact complex we mean a complex $\mathbb{X}$ such that  there exist $i_0$ such that
$\mathbb{X}^n = 0$ for $n \neq i_0, i_0 + 1$; also $\mathbb{X}^{i_0} = \mathbb{X}^{i_0 + 1} = R$ and $\partial^{i_0} = 1_R$. Note $\mathbb{X} = 0$ in $\Kc(\proj R)$. It is well known that if $\Gb$ is a complex of free $R$-modules then $\Gb = \Fb \oplus \mathbb{Y}$ where $\Fb$ is minimal and $\mathbb{Y}$ is a direct sum of simple split exact complexes.
Thus $\Gb \cong \Fb$ in $\Kc(\proj R)$.

We will need the following well-known result.
\begin{lemma}\label{hom-minimal}
(with hypotheses as above). If $\Fb, \Gb$ are minimal complexes of free $R$-modules. If $\Fb \cong \Gb$ in $\Kc(\proj R)$ then $\Fb \cong \Gb$ in $\Cc(\md(R))$.
\end{lemma}

\s Let $(A,\m)$ be a Gorenstein local ring. Let $\Kc_{ac}(\proj A)$ be the homotopy category of acyclic complexes of free $A$-modules. If $M$ is a MCM $A$-module then let $\Fb_M$ be a complete resolution of $M$. By \cite[4.4.1]{Bu} we have a triangle equivalence $\Theta \colon \CMS(A) \rt \Kc_{ac}(\proj A)$ with $\Theta(M) = \Fb_M$.
We give a proof of Theorem \ref{periodic-complex}. We restate it for the convenience of the reader.
\begin{theorem}\label{periodic-complex-exam}
Let $(R,\m)$ be a Noetherian local ring which is a quotient of a regular local ring. Then there exists a minimal periodic (of period $2s$)co-chain complex $\Fb$ of  free $R$-modules  (i.e., $\Fb \cong \Fb[2s]$) such that $H^i(\Fb)$ has finite length for all $i$. If $R$ is regular or if residue field of $R$ is infinite we can choose $s = 1$. In this case there exists a minimal complex $\Gb$ of free modules with $\Gb = \Gb[2]$.
\end{theorem}
\begin{proof}
  If $R$ is a singular complete intersection then it has an MCM two periodic module say $E$. The its complete resolution $\Fb_E$ does the job.

  So assume that $R$ is not a singular complete intersection.  We also first assume $R$ is NOT regular. By hypothesis  $R$ is a quotient of a regular local ring say $(Q, \q)$. Set $k = Q/\q$. We choose a regular sequence $f_1,\ldots, f_c \in \q^2$ such that $f_i \in \ann_Q R$ and $c = \dim Q - \dim R$. Set $A = Q/(f_1,\ldots, f_c)$. Note $A$ is a singular complete intersection.
It follows that
  \begin{enumerate}
    \item  $R$ is a quotient of $A$.
    \item $\dim R = \dim A$. Set $d = \dim A$.
    \item $R$ is not free as an $A$-module.
  \end{enumerate}
  Note we have a natural triangulated map $\Kc_{ac}(\proj A) \rt \Kc(\proj R)$ given by $-\otimes_A R$.
  Thus we have a triangulated map $\eta \colon \CMS^0(A) \rt \Kc(\proj(R)$ obtained by composing the previous map with the inclusion
  $\CMS^0(A)  \hookrightarrow \CMS(A) \cong \Kc_{ac}(A)$.
  As $R$ is not a free $A$-module we get that $\Tor_i^A(\Om_A^{d}(k), R) \neq 0$
for all $i \geq 1$. In particular
   $\eta(\Om_A^{d}(k)) \neq 0$. By Corollary \ref{vb-cor} there exists $0 \neq \Fb \in \Kc(\proj R))$ with
  $\Fb[2s] \cong \Fb$ for some $s \geq 1$ and $\Fb = \eta(X)$ for some $X$ in $\CMS^0(A)$. If $k$ is infinite then we can choose $s = 1$.
  We note that $H^i(\Fb) = \Tor^{A}_i(X, R)$ for $i \geq 1$. It follows that $\ell((H^i(\Fb))) < \infty$ for all $i$. Notice by construction $\Fb$ is a minimal complex.
  Therefore by \ref{hom-minimal} we get $\Fb \cong \Fb[2s] $ in $\Cc(\md(R))$.
  If $R$ is regular choose $A = R[X]/(X^2)$ and choose $X = \Om^{d}_A(k)$. Then same argument as before yields a complex $\Fb$ with $\Fb \cong \Fb[2]$.

  Furthermore if $\Fb[2] \cong \Fb$ in $\Cc(\md R)$ then we can by Proposition \ref{2-periodic} construct a minimal complex $\Gb$ with $\Gb = \Gb[2]$.
\end{proof}
We now state a result which we used in the previous Theorem.
\begin{proposition}\label{2-periodic}
Let  $(A,\m)$ be local and assume $\Fb \in \Cc(\proj A)$ with $\Fb \cong \Fb[2]$ and $H^i(\Fb)$ has finite length for all $i$. Then there exists $\Gb \in \Cc(\proj A)$ with
$\Gb = \Gb[2]$ and $H^i(\Gb)$ has finite length for all $i$.
\end{proposition}
\begin{proof}
  Let $\psi \colon \Fb \rt \Fb[2]$ be an isomorphism. We have a commutative diagram:
  \[
  \xymatrix
{
\Fb
\colon
\ \ar@{->}[r]
  & F^{0}
    \ar@{->}[d]^{\psi_0}
\ar@{->}[r]^{\alpha^0}
& F^{1}
    \ar@{->}[d]^{\psi_1}
\ar@{->}[r]^{\alpha^1}
 & F^{2}
 \ar@{->}[d]^{\psi_2}
 \ar@{->}[r]
 &\
 \\
 \Fb[2]
 \colon
 \ \ar@{->}[r]
   & F^{2}
\ar@{->}[r]^{\alpha^2}
& F^{3}
\ar@{->}[r]^{\alpha^3}
&F^4
\ar@{->}[r]
&\
 }
\]
Set $X = F^2$ and $Y = F^1$.
Also set $u = \alpha^0 \circ \psi_0^{-1}$ and $v = \alpha^1$.
Note
\begin{align*}
  v\circ u &= \alpha^1 \circ (\alpha^0 \circ \psi^{-1}_0) = (\alpha^1 \circ \alpha^0) \circ \psi^{-1}_0 = 0.\\
   u\circ v &= (\alpha^0 \circ \psi^{-1}_0) \circ \alpha ^1 = (\psi_1^{-1} \circ \alpha^2) \circ \alpha^1 = \psi_1^{-1} \circ ( \alpha^2 \circ \alpha^1) = 0.
   \end{align*}
   So we have a circular complex
   \[
   \Gb \colon  \cdots \xrightarrow{v} X \xrightarrow{u} Y \xrightarrow{v} X \xrightarrow{u} \cdots
   \]
   Thus $\Gb = \Gb[2]$.

   We note that $Z_Y = \ker \alpha^1$. As $\psi_0^{-1}$ is bijective we get $B_Y = \image \alpha^0$. It follows that $H_Y = H^1(\Fb)$ and so has finite length.
   Also note that $B_X = \image \alpha_1$. Also $u = \psi_1^{-1} \circ \alpha^2$. So $Z_X = \ker \alpha^2$. Thus  $H_X = H^2(\Fb)$ and so has finite length.
\end{proof}
\s\label{mambo} From now on we will work with minimal complexes $\Fb \in \Kc(\proj(A))$ with $\Fb[2] = \Fb$ and $H^i(\Fb)$ having finite length for all $i \in \Z$.
We first prove:
\begin{lemma}
  \label{rank}
  (with hypotheses as in \ref{mambo}). If $d = \dim A > 0$ then $\rank \Fb^i = \rank \Fb^{i+1}$ for all $i$.
\end{lemma}
\begin{proof}
Let $\q$ be a minimal prime of $A$ and let $B = A_\q$.
  We  write $\Fb$ as
  \[
  \cdots \rt  X \xrightarrow{\alpha} Y \xrightarrow{\beta} X \xrightarrow{\alpha} Y \rt \cdots
  \]
  As $\ell(H^i(\Fb)) < \infty $  for all $i$ we get that $\ker(\alpha)_\q = \image(\beta)_\q$ and
$\ker(\beta)_\q = \image(\alpha)_\q$.
We consider the complex
\[
\Cc\colon 0 \rt \ker(\beta) \rt Y \xrightarrow{\beta} X \rt \image(\alpha) \rt 0.
\]
Note $\Cc_\q$ is exact. Counting lengths we get  $\rank_B Y_\q = \rank_B X_\q$. The result follows.
\end{proof}

 \emph{Construction of periodic minimal complexes over \CM \ local rings of dimension $1,2,3$.}\\
We first consider the case when $\dim A = 1$.
\begin{construction}
($\dim A = 1$: Let $x \in \m$ be an $A$-regular element). Consider $\Fb_x$ where $\Fb_x^i = A$ for all $i \in \Z$ and
\begin{enumerate}
  \item $\partial^i = 0$ for $i$ even.
  \item $\partial^i = $ multiplication by  $x$ for $i$ odd.
\end{enumerate}
Then $H^i(\Fb_x) = A/(x)$ for $i$ even and $H^i(\Fb_x) = 0$ for $i$ odd.
\end{construction}

Next we consider the case when $\dim A = 2$.
\begin{construction}
($\dim A = 2$.) Let $x_1,x_2\in \m$ be an $A$-regular sequence and let $u_1,u_2 \in \m$ be another regular sequence.  Set $I = (x_1, x_2)$ and $J = (u_1, u_2)$. Consider $\Fb_{I, J}$ where $\Fb_{I, J}^i = A^2$ for all $i \in \Z$. For convenience set $\Fb_{I, J}^i = X$ for $i$ even with a basis $e_1, e_2$ and $\Fb_{I, J}^i = Y$ for $i$ odd with  a basis $f_1, f_2$. Also set
$p = u_2f_1 - u_1f_2 \in Y$ and $q = x_2e_1 - x_1e_2 \in X$.
\begin{enumerate}
  \item For $i$ even set $\partial^i(e_j) = x_jp$ for $j = 1,2$.
  \item For $i$ odd set $\partial^i(f_j) = u_jq$ for $j = 1,2$.
\end{enumerate}
It is easily verified that $\Fb_{I, J}$  is indeed a complex. To compute cohomology first note that if $\partial^0(ae_1 + be_2) = 0$ then $(ax_1 + bx_2)p = 0$. As $u_2$ is $A$-regular we get $ax_1 + bx_2 = 0$. As $x_1, x_2$ is a regular sequence we get that $ae_1 + be_2 \in Aq$. So $Z^0 = Aq$. Clearly $B^0 = Jq$. It follows that $H^0(\Fb_{I, J} ) \cong A/J$.
By symmetry we get $H^1( \Fb_{I, J})\cong A/I$.
\end{construction}

Finally we consider the case when $\dim A = 3$.
\begin{construction}
($\dim A = 3$.) Let $\bx = x_1,x_2, x_3\in \m$ be an $A$-regular sequence.
Consider $\Fb_{\bx}$ where $\Fb_{\bx}^i = A^3$ for all $i \in \Z$. For convenience set $\Fb_{\bx}^i = X$ for $i$ even with a basis $e_1, e_2,e_3$ and $\Fb_{\bx}^i = Y$ for $i$ odd with  a basis $f_1, f_2, f_3$. Also set
$p   =  -x_3e_1 + x_2e_2 - x_1e_3 \in X$.
\begin{enumerate}
  \item For $i$ odd set $\partial^i(f_j) = x_jp$ for $j = 1,2,3$.
  \item For $i$ even set $\partial^i(e_1) = -x_2f_1 + x_1f_2$, $\partial^i(e_2) = -x_3f_1 + x_1f_3$ and $\partial^i(e_3) = -x_3f_2 + x_2f_3$.
\end{enumerate}
\end{construction}
It is easily verified that $\Fb_{\bx}$  is indeed a complex. Also note the differentials are precisely as in the Koszul complex on $x_1, x_2, x_3$. It follows that $H^0(\Fb_\bx) \cong A/(\bx)$ and $H^1(\Fb_\bx) = 0$.

\s We know that if $\dim A > 0$ and $\Fb$ is a minimal complex of free $A$-modules with $\Fb = \Fb[2]$ and $\ell(H^i(\Fb)) < \infty $ for all $i \in \Z$ then $\rank{\Fb^i}$ is constant for all $i \in  \Z$. We set this constant value by $\beta(\Fb)$. For regular local rings we show the following:
\begin{theorem}\label{regular-periodic}
Let $(A,\m)$ be a regular local ring and let $\Fb \neq 0$ be a minimal complex of free $A$-modules with $\Fb = \Fb[2]$ and $\ell(H^i(\Fb)) < \infty $ for all $i \in \Z$.
\begin{enumerate}[\rm (1)]
  \item If $\dim A = 2$ then $\beta(\Fb) \geq 2$.
  \item If $\dim A = 3$ then $\beta(\Fb) \geq 3$.
  \item If $\dim A = 4$ then $\beta(\Fb) \geq 4$.
\end{enumerate}
\end{theorem}
To prove this Theorem we need the following:
\begin{lemma}\label{depth2}
Let $(A,\m)$ be a regular local ring with $\dim A \geq 2$ and let $\Fb \neq 0$ be a minimal complex of free $A$-modules with $\Fb = \Fb[2]$ and $\ell(H^i(\Fb)) < \infty $ for all $i \in \Z$.
Then
\begin{enumerate}[\rm (1)]
  \item $H^*(\Fb) \neq 0$.
  \item $\partial^i \neq 0$ for all $i$.
\end{enumerate}
\end{lemma}
\begin{proof}
(1) Suppose if possible $H^*(\Fb) = 0$. Then $\Fb \in \Kc_{ac}(\proj A)$. But \\  $\Kc_{ac}(\proj A) \cong \CMS(A)$. As $A$ is regular local, any MCM $A$-module is free. So $\CMS(A) = 0$. Therefore $\Fb$ is contractible. As $\Fb$ is minimal, it can be easily shown that if $\Fb$ is contractible then $\Fb = 0$, a contradiction.

  (2) Suppose if possible $\partial^{-1} = 0$. If $\ker \partial^0 \neq 0$ then it has positive depth. It follows that $\depth H^0(\Fb) > 0$ a contradiction. So $\ker \partial^0 = 0$.  So $H^0(\Fb) = 0$.  By (1) we have that $H^1(\Fb) \neq 0$. We have an exact sequence
  \[
  0 \rt \Fb^0 \xrightarrow{\partial^0} \Fb^1 \rt H^1(\Fb) \rt 0.
  \]
  By depth Lemma we get $\depth H^1(\Fb) > 0$, a contradiction.  Thus $\partial^{-1} \neq 0$. By symmetry $\partial^0 \neq 0$.
\end{proof}

We now give
\begin{proof}[Proof of Theorem \ref{regular-periodic}]
  (1) Suppose if possible there exists $\Fb$ with $\beta(\Fb) = 1$.
  Then we have
  \[
  \Fb \colon \cdots \rt A \xrightarrow{x} A \xrightarrow{y}  A \xrightarrow{x} A \xrightarrow{y} \rt \cdots
  \]
  Then $xy = 0$. As $A$ is a domain we get $x = 0$ or $y = 0$. Say $y = 0$. Then $H^1(\Fb) = A/(x)$ which does not have finite length as $\dim A = 2$; which is a contradiction.

  (2) Let $x$ be a regular parameter. Then $A/(x)$ is regular local of dimension $2$. Note $\Fb/x \Fb$ satisfies the hypothesis of (1). It follows that $\beta(\Fb) = \beta(\Fb/x \Fb) \geq 2$. We prove that $\beta(\Fb) = 2$ is not possible.

  Suppose if possible  there exists $\Fb$ with $\beta(\Fb) = 2$.\\ By Lemma \ref{depth2} we get that $\image \partial^i \neq 0$ for all $i$. So $\rank \image \partial^i = 1$ for all $i$. As $\ell(H^i(\Fb)) < \infty $ for all $i$ we get that $\rank \ker \partial^i = 1$ for all $i$. Thus $\ker \partial^i $ is isomorphic to an ideal $I_i$ in $A$. Note this isomorphism maps $\image \partial^{i-1} $ to an ideal $J_{i} \subseteq I_i$ of $A$. We note that $J_0$ is generated by $\leq 2$ elements. Say $J_0 = (u,v)$.

  Claim-(i) $H^i(\Fb) \neq 0$ for all $i \in \Z$. \\
  Suppose if possible $H^0(\Fb) = 0$. By Lemma \ref{depth2} we get $H^{-1}(\Fb) \neq 0$. We have an exact sequence
  \[
  0 \rt I_{-1} \rt \Fb^{-1} \xrightarrow{ \partial^{-1}} \Fb^0 \rt \image \partial^0 \rt 0.
  \]
  By depth Lemma we get $I_{-1}$ is free and so $\cong A$. As $J_{-1}$ is generated by $2$ elements we get that $H^{-1}(\Fb)$ is \emph{not} of finite length; a contradiction.
  Thus $H^0(\Fb) \neq 0$. Similarly $H^1(\Fb) \neq 0$. As $\Fb$ is $2$-periodic we get $H^i(\Fb) \neq 0$ for all $i \in \Z$.

  As $A$ is regular local it is in particular a UFD. Suppose $u = au_1$ and $v = av_1$ where $u_1$ and $v_1$ do not have any common factors. \\
  Claim-(ii) $\m \notin \Ass A/J_0$. \\
  If we prove claim (ii) then result follows from claim (i) as $H^0(\Fb)$ is a submodule of $A/J_0$.
  If $u_1 =1$ or $v_1 =1$ then $J_0 = (a)$ and clearly this implies Claim (ii).
  So $u_1 \neq 1$ and $v_1 \neq 1$. Notice $\htt(u_1, v_1) \geq 2$ and so $= 2$. Thus $u_1, v_1$ is an $A$-regular sequence. If $a = 1$ then also clearly the result follows.
Thus also $a \neq 1$. For $p,q \in A$  we write $p\mid q$ if $p$ divides $q$ and $p \nmid q$ otherwise. Choose a prime  element $p \in \m$ such that $p \nmid a$. Suppose if possible $\m \in \Ass A/J_0$. Say $\m = (0 \colon \ov{t})$ where $t \notin J_0$.
So $pt \in J_0 = (au_1, av_1)$. It follows that $a \mid t$. Say $t = at_1$. Then check that $\m t_1 \subseteq (u_1, v_1)$. As $u_1, v_1$ is an $A$-regular sequence we get $t_1 \in (u_1, v_1)$. This implies that $t = at_1 \in J_0$ which is a contradiction. So claim (ii) holds and as discussed before this proves our result.

(3)  Let $x$ be a regular parameter. Then $A/(x)$ is regular local of dimension $3$. Note $\Fb/x \Fb$ satisfies the hypothesis of (2). It follows that $\beta(\Fb) = \beta(\Fb/x \Fb) \geq 3$. We prove that $\beta(\Fb) = 3$ is not possible.

Suppose if possible  there exists $\Fb$ with $\beta(\Fb) = 2$.\\
As all cohomology modules have finite length we have
\[
\rank \image \partial^{i-1} + \rank \ker \partial^{i-1} = \rank \ker \partial^i + \rank \image \partial^i = 3.
\]
By Lemma \ref{depth2} we get that $\image \partial^i \neq 0$ for all $i$. So without any loss of generality we may assume $\rank \image \partial^{-1} = 1$.  As $H^0(\Fb)$ has finite length we get $\rank \ker \partial^0 = 1$.  Thus $\ker \partial^0 $ is isomorphic to an ideal $I$ in $A$. Note this isomorphism maps $\image \partial^{-1} $ to an ideal $J \subseteq I$ of $A$. We note that $J$ is generated by $\leq 3$ elements. Say $J = (u_1, u_2, u_3)$. As $A$ is a UFD we may choose $a$ which is the  greatest common divisor of $u_1, u_2,u_3$. Say $u_i = av_i$  for $i = 1, 2,3$. Set $K = (v_1, v_2,v_3)$. Then $J = aK$ and either $K = A$ or $ 2 \leq \htt K \leq 3$ (the first inequality is easily seen by taking a primary decomposition of $K$).

Claim (iii) $K = A$. \\
Assume the claim for the time being. Then $J \cong A$ is free. So $\depth J = 4$. As $\depth I  \geq 2$ we get that if $H^0(\Fb) \neq 0$ then $\depth H^0(\Fb) \geq 1$; which is a contradiction. So $H^0(\Fb) = 0$. So $I = J \cong A$. By  Lemma \ref{depth2} we get $H^1(\Fb) \neq 0$.  We have an exact sequence
\[
0 \rt I \rt \Fb^0 \rt \image \partial^0 \rt 0.
\]
Counting depths we have $\depth \image \partial^0 \geq 3$. Also $\depth \ker \partial^1 \geq 2$. So \\  $\depth H^1(\Fb) \geq 2$ which is a contradiction. Thus it suffices to prove Claim (iii).

\emph{Proof of Claim (iii):} Suppose if possible $K \neq A$. Let $P$ be a minimal prime of $K$. We have $2 \leq \htt P \leq 3$. Note $\Fb_P$ is exact. As $A_P$ is regular local we get $J_P$ is free. In particular $J_P$ is principal. But $J_P = (a)_P K_P$ is definitely not principal as $\htt K_P \geq 2$. So claim (iii) is true and as shown earlier this implies our result.
\end{proof}

\begin{remark}
  We do not know an example of a $2$-periodic complex $\Fb$ over a regular local ring of dimension $4$ such that $\beta(\Fb) = 4$.
\end{remark}
Finally we show:
\begin{theorem}
\label{exp}
Let $(A,\m)$ be a regular local ring of dimension $d \geq 1$. Then there exists a $2$-periodic minimal complex $\Fb$ with finite length cohomology such that
$\beta(\Fb) = 2^d$.
\end{theorem}
\begin{proof}
  Set $(B,\n) = (A[Y]/(Y^2), (\m, Y))$ and $k = A/\m = B/\n$.
  Let $\Gb $ be a minimal complete resolution of $\Om^d_B(k)$. Then same argument as in proof of Theorem \ref{periodic-complex-exam} shows that $\Fb = \Gb \otimes_B A$ is a minimal $2$-periodic complex with finite length cohomology. Thus it suffices to prove $\beta(\Fb) = \beta(\Gb) = 2^d$. Equivalently it suffices to prove that $\beta_i(\Om^d_B(k)) = 2^d$ for all $i$. As $\Om^d_B(k)$ does not have a free module as a direct summand (see \cite[1.3]{D})
it suffices to show $\mu(\Om^d_B(k)) = \beta_d^B(k) = 2^d$.

   Let $\bx = x_1,\ldots, x_d$ be a regular system of parameters of $A$. Then note that $\bx$  considered in $B$ is a regular sequence in $\n \setminus \n^2$. Set $C = B/\bx B$.
   By \cite[5.3]{T} we get
   \[
   \beta_d^B(k)  = \sum_{i= 0}^{d} \binom{d}{i} \beta_i^C(k).
   \]
   Now notice $C = k[Y]/(Y^2)$. Then the minimal resolution of $k$ over $C$ is given as:
   \[
   \cdots \rt C \xrightarrow{Y}  C \xrightarrow{Y}  C \xrightarrow{Y} k \rt 0.
   \]
   So $\beta_i^C(k) = 1$ for all $i$. It follows that $\beta_d^B(k) = 2^d$.
\end{proof}

\end{document}